\documentclass{birkjour}
\usepackage{amsmath}
\usepackage{amssymb}
\usepackage{amsthm}
\usepackage{hyperref}
\numberwithin{equation}{section}
\newtheorem{theorem}{Theorem}[section]
\newtheorem{lemma}[theorem]{Lemma}
\newtheorem{proposition}[theorem]{Proposition}

\newtheorem{definition}[theorem]{Definition}

\newtheorem{remark}[theorem]{Remark}

\newtheorem{main}[theorem]{Main Results}
\def\H{H_{\sigma}}
\def\dH{\dot{H}_{\sigma}}
\def\L{L_{\sigma}}
\begin{document}
\title[Semi-dissipative and Ideal Boussinesq Equations] {Higher Order Regularity and Blow-up Criterion for Semi-dissipative and Ideal Boussinesq Equations}

\author[Utpal Manna]{Utpal Manna}

\address{%
School of Mathematics\\
Indian Institute of Science Education and Research (IISER) Thiruvananthapuram\\
Thiruvananthapuram 695016\\
Kerala, INDIA}

\email{manna.utpal@iisertvm.ac.in}

\author[Akash A. Panda]{Akash A. Panda}

\address{%
School of Mathematics\\
Indian Institute of Science Education and Research (IISER) Thiruvananthapuram\\
Thiruvananthapuram 695016\\
Kerala, INDIA}

\email{akash.panda13@iisertvm.ac.in}



\begin{abstract}
In this paper we establish local-in-time existence and uniqueness of strong solutions in $H^s$ for $s > \frac{n}{2}$ to the viscous, zero thermal-diffusive Boussinesq equations in $\mathbb{R}^n , n = 2,3$.  Beale-Kato-Majda type blow-up criterion has been established in three-dimensions with respect to the $BMO$-norm of the vorticity. We further prove the local-in-time existence and blow-up criterion for non-viscous and fully ideal Boussinesq systems. Commutator estimates due to Kato and Ponce (1988) \cite {KP} and Fefferman et. al. (2014) \cite {Fe} play important roles in the calculations.

\end{abstract}

\maketitle\setcounter{equation}{0}

\maketitle

\section{Introduction}

The Boussinesq equations for the incompressible fluid flows interacting with an active scalar under the external potential force are given by: 
\begin{align}
 \frac{\partial \mathbf u}{\partial t} + (\mathbf u \cdot \nabla)\mathbf u - \nu \Delta \mathbf u + \nabla p = \mathbf \theta f, \ \ \textrm{in}\ \mathbb{R}^n\times (0, \infty),
\end{align}
\begin{align}
\frac{\partial \mathbf \theta}{\partial t} + (\mathbf u \cdot \nabla)\mathbf \theta = \kappa \Delta \mathbf \theta,  \ \ \textrm{in}\ \mathbb{R}^n\times (0, \infty),
\end{align}
\begin{align*}
\nabla \cdot \mathbf u = 0,  \ \ \textrm{in}\ \mathbb{R}^n\times (0, \infty),
\end{align*}
\begin{align*}
\mathbf u(x,0) = \mathbf u_{0}(x) , \mathbf \theta(x,0) = \mathbf \theta_{0}(x ),  \ \ \textrm{in}\ \mathbb{R}^n,
\end{align*}
where $n$ = 2, 3. Here $\mathbf u$ = $\mathbf u(x,t)$ = $(\mathbf u_{1}(x,t), \dots, \mathbf u_{n}(x,t)),$  a map from  $\mathbb R^{n}$ $\times$ [0, $\infty$) $\to$ $\mathbb R^{n}$, is a vector field denoting the velocity; $\mathbf \theta = \mathbf\theta(x, t)$, a map from $\mathbb R^{n}$ $\times$ [0, $\infty$) $\to$ $\mathbb R$, is a scalar function denoting the temperature of the fluid; $p = p(x, t)$ is the scalar pressure; $\nu\geq 0$ is the kinematic viscosity; $\kappa\geq 0$ is the thermal diffusivity; $(\mathbf u \cdot \nabla)\mathbf u$ has $j$-th component $\mathbf {u}_k \partial_k \mathbf {u}_j$ (with usual summation convention); and  $f = f(x, t)$ is the external potential force i.e. $\nabla\times f = 0.$ We denote the above system by $B_{\nu, \kappa}$ when $\nu > 0$ and $\kappa > 0$. 

The global-in-time regularity in two-dimensions of the system $B_{\nu, \kappa}$ is known for a long time \cite{CD},  and in three-dimensions one has local existence of weak solutions, much like the Navier-Stokes equations. Due to the parabolic-parabolic coupling, it is indeed possible to rewrite the above system in the abstract framework of the Navier-Stokes equations and then use the standard solvability techniques (e.g. see Temam \cite{Te}). In three-dimensions, due to coupling with the Navier-Stokes equations, one can at-most obtain the local-in-time solvability result with arbitrary initial data and global-in-time result for sufficiently small initial data. 

On the other hand, three interesting cases arise when we consider the Boussinesq systems $(i) B_{\nu, 0},$ i.e. when $\nu> 0, \kappa=0$ (parabolic-hyperbolic coupling), $(ii) B_{0, \kappa},$ i.e., when $\nu=0, \kappa > 0$ (hyperbolic-parabolic coupling) and $(iii) B_{0, 0},$ i.e., when $\nu=0=\kappa$ (hyperbolic-hyperbolic coupling). 

One of the main mathematical motivations for studying the above three cases is that the semi-dissipative and the ideal Boussinesq systems are among the most commonly used models to understand the vortex stretching effect for three-dimensional incompressible flows. There are similarities between the three-dimensional Euler equations and the Boussinesq equations with respect to the problem of breakdown of smooth solutions in finite time. In the celebrated paper of Beale-Kato-Majda \cite{BKM}, it was shown that if a solution to the three-dimensional Euler equations is smooth at the initial stage and loses its regularity after some finite time, then the bound on $L^{\infty}$-norm of vorticity inhibits the breakdown of smooth solutions. 

In two-dimensions, well-posedness of the Boussinesq system in the above three cases has been studied by several authors under different assumptions on the initial data (e.g. see \cite{AH, Dc, CKN, CN, DP, DP1, HK, HL}, to name a few). It is worth to note (see \cite{Dc}) that the systems $B_{\nu, 0}$ and $B_{0, \kappa}$ have unique global solutions in two-dimensions provided the initial data $(\mathbf{u_0}, \mathbf{\theta}_0)\in\H^{m}(\mathbb R^{2})\times H^{m}(\mathbb R^{2})$ with $m\geq 3$, $m$ an integer. Similar result for the system $B_{\nu, 0}$ was obtained in \cite{HL} under lower regularity condition on the initial data for temperature, i.e. when $(\mathbf{u_0}, \mathbf{\theta}_0)\in\H^{m}(\mathbb R^{2})\times H^{m-1}(\mathbb R^{2})$ with $m\geq 3$. Global well-posedness results for $B_{\nu, 0}$ with initial data in suitable Besov and Lorentz spaces are due to \cite{HK} and \cite{DP, DP2} respectively. On the other hand, for the system $B_{0, 0}$ only local-in-time existence results are available even in two-dimensions. It was proved in \cite{CN} that if the initial data $(\mathbf{u_0}, \mathbf{\theta}_0)\in\H^{3}(\mathbb R^{2})\times H^{3}(\mathbb R^{2})$, then local-in-time classical solutions exist and is unique. Moreover,  Beale-Kato-Majda type criterion for blow-up of smooth solutions is established in \cite{CN}. More precisely, they proved that the smooth solution exists on $[0, T]$ if and only if $\nabla\mathbf\theta\in L^1 (0, T; L^{\infty}(\mathbb{R}^2))$.

In three-dimensions, a very few local-in-time existence results and blow-up criterion are available (e.g. see \cite{GF, IM, QDY, Ye}). However, in the very particular case of the axisymmetric initial data, global-in-time well-posedness has been proven in three-dimensions by  Abidi et. al. \cite{AHK}.

In this work, we consider the Boussinesq systems $B_{\nu, 0}$, $B_{0, \kappa}$ and  $B_{0, 0}$ in both two and three dimensions and prove local-in-time existence and uniqueness of the strong solutions when the initial data $(\mathbf{u_0}, \mathbf{\theta}_0)\in\H^{s}(\mathbb R^{n})\times H^{s}(\mathbb R^{n})$, where $s>n/2$ for $B_{\nu, 0}$, and $s>n/2+1$ for both $B_{0, \kappa}$ and  $B_{0, 0}, \ \ n=2, 3, s\in\mathbb{R}$. We, however, do not establish the global well-posedness for the systems $B_{\nu, 0}$, and $B_{0, \kappa}$ in two-dimensions as this result is known in the literature due to Chae\cite{Dc}. We thereafter investigate the blow-up criteria in two-dimensions for the system $B_{0, 0}$ and in three-dimensions for the systems $B_{\nu, 0}$, $B_{0, \kappa}$ and  $B_{0, 0}$. We prove when $s>n/2+1$, $BMO$-norm of the vorticity and that of the gradient of the temperature (i.e. $\nabla\times\mathbf u, \nabla\mathbf\theta \in L^1(0, T; BMO)$) control the breakdown of smooth solutions of the above systems.  However, we later show that under suitable additional assumption on $\mathbf{\theta}_0$, one can completely relax the condition on gradient of the temperature and the condition $\nabla\times\mathbf u\in L^1(0, T; BMO)$ alone can ensure that the smooth solution persists.

To be precise, we will prove the following main results in this paper.

\begin{main}\label{mt1}
Let $s\in\mathbb{R}$ and $\nabla\cdot\mathbf{u_0}=0$ and  $f\in L^{\infty} \left([0, T] ; H^s(\mathbb R^n)\right), \ n = 2, 3.$  
\begin{itemize}
\item If $(\mathbf{u_0}, \mathbf{\theta}_0)\in\H^{s}(\mathbb R^{n})\times H^{s}(\mathbb R^{n}),\ s>\frac{n}{2}$, then there exists a unique strong solution $(\mathbf{u}, \mathbf{\theta})$ to the Boussinesq system $B_{\nu, 0}$ with the regularity $\mathbf{u}\in C([0, \tilde{T}]; \H^{s}(\mathbb R^{n})) \cap L^{2}([0, \tilde{T}];  \H^{s+1}(\mathbb R^{n}))$ and $\mathbf{\theta} \in C([0, \tilde{T}]; H^{s}(\mathbb R^{n}))$ for some finite time $\tilde{T} = \tilde{T}(s, \nu, \|u_0\|_{\H^{s}}, \|\theta_0\|_{H^{s}}) > 0$.
\vskip.1in
\item If $(\mathbf{u_0}, \mathbf{\theta}_0)\in\H^{s}(\mathbb R^{n})\times H^{s}(\mathbb R^{n}),\ s>\frac{n}{2}+1$, then there exists a unique strong solution $(\mathbf{u}, \mathbf{\theta})$ to the Boussinesq system $B_{0, \kappa}$ with the regularity $\mathbf{u}\in C([0, \tilde{T}]; \H^{s}(\mathbb R^{n}))$ and $\mathbf{\theta} \in C([0, \tilde{T}]; H^{s}(\mathbb R^{n}))\cap L^{2}([0, \tilde{T}];  H^{s+1}(\mathbb R^{n}))$ for some finite time $\tilde{T} = \tilde{T}(s, \kappa, \|u_0\|_{\H^{s}}, \|\theta_0\|_{H^{s}}) > 0$.
\vskip.1in
\item If $(\mathbf{u_0}, \mathbf{\theta}_0)\in\H^{s}(\mathbb R^{n})\times H^{s}(\mathbb R^{n}),\ s>\frac{n}{2}+1$, then there exists a unique strong solution $(\mathbf{u}, \mathbf{\theta})$ to the Boussinesq system $B_{0, 0}$ with the regularity $\mathbf{u}\in C([0, \tilde{T}]; \H^{s}(\mathbb R^{n}))$ and $\mathbf{\theta} \in C([0, \tilde{T}]; H^{s}(\mathbb R^{n}))$ for some finite time $\tilde{T} = \tilde{T}(s, \|u_0\|_{\H^{s}}, \|\theta_0\|_{H^{s}}) > 0$.
\end{itemize}
Moreover, if $s > \frac{n}{2} + 1, f=e_n$ and $(\mathbf{u_0}, \mathbf \theta_0) \in \H^{s}(\mathbb R^{n})\times \left(H^s(\mathbb R^{n}) \cap W^{1, p}(\mathbb R^{n})\right),$ $2 \leq p \leq \infty$, then the condition
\[\int_{0}^{\tilde T} \|\nabla \times \mathbf{u}(\tau)\|_{BMO} \,d\tau < \infty\]
ensures that the solution $(\mathbf{u, \theta})$  to all the three Boussinesq systems $B_{\nu, 0}$ (in three-dimensions), $B_{0, \kappa}$ (in three-dimensions) and $B_{0, 0}$ (in both two and three-dimensions)  can be extended continuously to $[0, T]$ for some $T > \tilde T.$
\end{main}

\noindent 
Construction of the paper is as follows. After giving definitions and properties of various function spaces and inequalities in the next section, we start investigating about the Boussinesq system $B_{\nu, 0}$ in Section 3. 
At first, we consider the Fourier truncated system $B_{\nu, 0}$ on the whole of $\mathbb{R}^n$, and show that the solutions $(\mathbf u^{\textit R}, \mathbf \theta^{\textit R})$ of some smoothed version of  the Boussinesq system $B_{\nu, 0}$ exist. In subsection \ref{EM}, we show that the $H^s$-norm of $(\mathbf u^{\textit R}, \mathbf \theta^{\textit R})$ are uniformly bounded up to a terminal time $\tilde T$ which is independent of $R$. In the following subsection \ref{EU}, we prove that up to the blowup time, the solution $(\mathbf u^{\textit R}, \mathbf \theta^{\textit R})$ is a Cauchy sequence in the $L^2$-norm as $R \to \infty$, and by using Sobolev interpolation, $(\mathbf u^{\textit R}, \mathbf \theta^{\textit R}) \to (\mathbf{u}, \mathbf{\theta})$ in any $H^{s'}$ for $0 < s' < s$. Finally we provide the proof of the local-in-time existence and uniqueness of the strong solution in Theorem \ref{mt1}.	In subsection \ref{BC}, we establish that the $BMO$ norms of the vorticity and gradient of temperature control the breakdown of smooth solutions. In the later part of that subsection, we prove that the condition on the gradient of temperature can be completely relaxed under suitable assumption on the regularity of the initial temperature. Commutator estimates due to Kato and Ponce \cite{KP} play crucial role in this part of the calculations.  Finally in Section 4, we focus on the other two interesting Boussinesq systems, i.e. on  $B_{0, \kappa}$ and $B_{0, 0}$, and establish similar results (see Theorems \ref{NVBE1}, \ref{NVBE2}, \ref{IBE1}, \ref{IBE2}) as obtained for the system $B_{\nu, 0}$ in Section 3.  

\section{Mathematical framework}

We consider the evolution for positive times and the spatial variable belongs to $\mathbb R^{n}$ for $n = 2, 3.$

\subsection{Fractional Derivative Operators.}

For $s \in \mathbb R,$ let $\Lambda^{s}$ denote the Riesz potential of order $s$ which is equivalent to the operator $(-\Delta)^{s/2}$, where $\Delta$ is the Laplace operator and the fractional power is defined using Fourier transform.

\noindent
We define this fractional derivative operator in terms of Fourier transformation as follows
\[ {\mathcal F} \left[ \Lambda^{s} f \right](\xi) = |\xi|^{s} \widehat f(\xi) .\]

\noindent
The operator $\Lambda^{s}$ is also known as the Zygmund operator.

\noindent
The norm on  Homogeneous Sobolev space, i.e. $\dot{H}^{s}(\mathbb R^{n})$, is given by
\begin{align}\label{hsd}
\| f\|_{\dot{H}^{s}} = \left( \int_{\mathbb R^{n}} \left[ |\xi|^s |\widehat f(\xi)| \right]^2 \right)^{1/2} = \left\| |\xi|^s \widehat f \right\|_{L^2} = \| \Lambda^s f\|_{L^2}.
\end{align} 

\noindent
and the inner product on $\dot{H}^{s}(\mathbb{R}^n)$ is given by
\begin{align*}
(f, g)_{\dot{H}^{s}} = \left( |\xi|^s \widehat f(\xi), |\xi|^s \widehat g(\xi) \right)_{L^2} &= ({\mathcal F} \left[ \Lambda^s f \right](\xi), {\mathcal F} \left[ \Lambda^s g \right](\xi))_{L^2} \nonumber \\ &= \left( \Lambda^s f, \Lambda^s g \right)_{L^2}.
\end{align*}

Let $J^s$ denote the Bessel potential of order $s$ which is equivalent to the operator $(I - \Delta)^{s/2}$. We define it in terms of Fourier transformation as follows:
\[ {\mathcal F} \left[ J^{s} f \right](\xi) = (1 + |\xi|^2)^{s/2} \widehat f(\xi) .\]

\noindent
The norm on ${H}^{s}(\mathbb R^{n})$ is given by
\begin{align}\label{hs}
\| f\|_{H^{s}} = \left( \int_{\mathbb R^{n}} \left[ (1 + |\xi|^2)^{s/2} |\widehat f(\xi)| \right]^2 \right)^{1/2} &= \left\| (1 + |\xi|^2)^{s/2} \widehat f(\xi)\right\|_{L^2} \nonumber \\ &= \| J^s f\|_{L^2},
\end{align}

\noindent
and the inner product on $H^{s}(\mathbb R^{n})$ is given by
\begin{align*}
(f, g)_{H^{s}} = \left( (1 + |\xi|^2)^{s/2} \widehat f(\xi), (1 + |\xi|^2)^{s/2} \widehat g(\xi) \right)_{L^2} &= ({\mathcal F} \left[ J^s f \right](\xi), {\mathcal F} \left[ J^s g \right](\xi))_{L^2} \nonumber \\ &= \left( J^s f, J^s g \right)_{L^2}.
\end{align*}

Clearly, we have the inclusion $H^{s}(\mathbb R^{n}) \subset \dot{H}^{s}(\mathbb R^{n})$, and $H^s = \dot{H}^{s} \cap L^2.$

\begin{remark}\label{gradhs}
One of the basic differences in properties of the above defined operators is the following:
\[\| \Lambda^{s-1} \nabla f\|_{L^2} = \left\| |\xi|^{s-1} \widehat {\nabla f}(\xi) \right\|_{L^2} = \left\| |\xi|^{s-1}\left. i\xi\right. \widehat f(\xi)\right\|_{L^2} = \left\| |\xi|^{s} \widehat f(\xi)\right\|_{L^2} = \| \Lambda^s f\|_{L^2}.\]
Therefore,
\[ \| \nabla f\|_{\dot{H}^{s-1}} = \|f\|_{\dot{H}^{s}}.\]
Whereas,
\begin{align*}
\| J^{s-1} \nabla f\|_{L^2} = \left\| (1 + |\xi|^2)^{\frac{s-1}{2}} \widehat {\nabla f}(\xi) \right\|_{L^2} &= \left\| (1 + |\xi|^2)^{\frac{s-1}{2}}\left. i\xi\right. \widehat f(\xi)\right\|_{L^2} \\ &\leq \left\| (1 + |\xi|^2)^{\frac{s}{2}} \widehat f(\xi)\right\|_{L^2} = \| J^s f\|_{L^2},
\end{align*}
implies
 \[\| \nabla f\|_{{H}^{s-1}} \leq \|f\|_{{H}^{s}}.\]
\end{remark}

\begin{remark}\label{prodhs}
If $s > n/2$, then each $f \in H^{s}(\mathbb R^{n})$ is bounded and continuous and hence 
\[ \|f\|_{L^{\infty}(\mathbb R^{n})} \leq C \|f\|_{H^s(\mathbb R^{n})}, \left. \right. for\left. \right. s > n/2.\]
Also, note that $H^s$ is an algebra for $s > n/2$, i.e., if $f, g \in H^{s}(\mathbb R^{n})$, then $fg \in H^{s}(\mathbb R^{n})$, for $s > n/2$. Hence, we have
\[ \| fg\|_{H^s} \leq C \| f\|_{H^s} \| g\|_{H^s}, \left. \right. for\left. \right.s > n/2.\]
\end{remark}

We define the spaces
\[L^{2}_{\sigma}(\mathbb R^{n}) = \left\{ f \in L^2 (\mathbb R^{n}) : \nabla \cdot f = 0 \right\}.\] and \[H^{s}_{\sigma}(\mathbb R^{n}) = \left\{ f \in H^{s}(\mathbb R^{n}) : \nabla \cdot f = 0 \right\}.\]

\begin{remark}\label{div}
Fix $s > n/2$ and let $f \in H^{s}_{\sigma}$ and $g \in H^s$. Then
\[ \| (f \cdot \nabla) g \|_{H^{s-1}} \leq C \|f\|_{\H^s} \|g\|_{H^s}.\]
\end{remark}

\begin{proof}
Being in $H^{s}_{\sigma},$ $f$ is divergence free, $(f \cdot \nabla) g = \nabla \cdot(f \otimes g)$.
And $H^s$ is an algebra for $s >n/2$,
\[\| (f \cdot \nabla) g \|_{H^{s-1}} = \| \nabla \cdot(f \otimes g)\|_{H^{s-1}} \leq C \| f \otimes g\|_{H^s} \leq C \|f\|_{\H^s} \|g\|_{H^s}.\]
\end{proof}

\begin{lemma}\label{Si}
$($Sobolev Inequality$)$ For $f \in H^{s}(\mathbb R^{n})$, we have
\[ \|f\|_{L^{q}(\mathbb R^{n})} \leq C_{n, s, q} \|f\|_{H^s(\mathbb R^{n})}\]
provided that q lies in the following range
\begin{enumerate}
\item[(i)]
if $s < n/2$, then $2 \leq q \leq \frac{2n}{n - 2s}$.
\item[(ii)]
if $s = n/2$, then $2 \leq q < \infty$.
\item[(iii)]
if $s > n/2$, then $2 \leq q \leq \infty$.
\end{enumerate}
\end{lemma}
\noindent The above lemma is a special case of the known generalized result  (e.g. see Theorem 2.4.5 of \cite{Ks}). 
\begin{remark}\label{r23}
We deduce the following result using Lemma \ref{Si}. For n = 2, we use  H\"{o}lder's inequality with exponents 2/$\epsilon$ and 2/(1-$\epsilon$), and Sobolev inequality for 0 $<$ $\epsilon$ $<$ s-1 to obtain
\[ \|fg\|_{L^2} \leq \|f\|_{L^{2/ {\epsilon}}}\|g\|_{L^{2/1-{\epsilon}}} \leq C \|f\|_{{\dot{H}}^{1-{\epsilon}}} \|g\|_{{\dot{H}}^{\epsilon}} \leq C \|f\|_{H^{1-{\epsilon}}} \|g\|_{H^{\epsilon}} \leq C \|f\|_{H^1} \|g\|_{{H}^{s-1}}.\]
For n = 3, we use H\"{o}lder's inequality with exponents 6 and 3, and Sobolev inequality to obtain
\[ \|fg\|_{L^2} \leq \|f\|_{L^6} \|g\|_{L^3} \leq C \|f\|_{{\dot{H}}^{1}} \|g\|_{{\dot{H}}^{1/2}} \leq C \|f\|_{H^1} \|g\|_{H^{1/2}} \leq C \|f\|_{H^1} \|g\|_{{H}^{s-1}}.\]
Note that, for both the two and three dimensions, we obtain the same bounds.
\end{remark}

\begin{lemma}\label{iss}
$($Interpolation in Sobolev spaces$)$. Given $s > 0,$ there exists a constant C depending on s, so that for all $f \in H^{s}(\mathbb R^{n})$ and $0 < s' < s,$
\[\|f\|_{H^{s'}} \leq C \|f\|_{L^2}^{1-s'/s}\|f\|^{s'/s}_{H^{s}}.\]
\end{lemma}

\noindent For details see \cite{AF} and for proof see Theorem 9.6, Remark 9.1 of \cite{LM}.

\subsection{Fourier Truncation Operator.}\label{FTO}

Let us define the Fourier truncation operator $\mathcal S_{\textit{R}}$ as follows:
\[ \widehat{{\mathcal S_{\textit{R}}} f}(\xi) := \mathbf 1_{B_{R}}(\xi)\widehat f(\xi),\]
where $B_R$, a ball of radius $R$ centered at the origin and  $\mathbf 1_{B_{R}}$ is the indicator function. Then we infer the following important properties:

\begin{enumerate}
\item 
$\|{\mathcal S_{\textit{R}}} f \|_{H^{s}(\mathbb R^{n})} \leq C \|f \|_{H^{s}(\mathbb R^{n})}.$

where $C$ is a constant independent of $R$.
\item
$\| {\mathcal S_{\textit{R}}} f  - f\|_{H^{s}(\mathbb R^{n})} \leq  \frac{C}{R^{k}} \|f \|_{H^{s+k}(\mathbb R^{n})}.$
\item
$\| ({\mathcal S_{\textit{R}}} - {\mathcal S_{\textit R'}}) f\|_{H^{s}} \leq C \left. \max\left\{ \left( \frac{1}{R} \right)^k , \left( \frac{1}{R'} \right)^k \right\} \right. \|f \|_{H^{s+k}}.$
\end{enumerate}
\noindent For the details of the proof of the properties see \cite{Fe}.

\subsection{Commutator Estimates.}

Let $f$ and $g$ are Schwartz class functions, then for $s \geq 0$, we define,
\[[J^s, f]g = J^s(fg) - f(J^{s}g),\]
and
\begin{align}\label{cef}
[J^s, f] \nabla g = J^s[(f \cdot \nabla)g] - (f \cdot \nabla)(J^{s}g),
\end{align}
where $[J^s, f] = J^{s}f - f J^{s}$ is the commutator, in which $f$ is regarded as a multiplication operator. Then we have the following celebrated commutator estimate due to Kato and Ponce \cite{KP}.
\begin{lemma}
 For $s \geq 0,$ and $1 < p < \infty$,
\[\|[J^s, f]g\|_{L^p} \leq C\left( \|\nabla f\|_{L^{\infty}}\|J^{s-1}g\|_{L^p} + \|J^{s}f\|_{L^p}\|g\|_{L^{\infty}}\right),\]
where C is a constant depending only on $n, p, s.$
\end{lemma}
\noindent For proof see the appendix of \cite{KP}.\\
\noindent
In particular, for $p=2$, $s>n/2$ and $(\mathbf u$, $\mathbf \theta)$ $\in$ $\H^s(\mathbb R^n) \times H^{s}(\mathbb R^{n})$, we have
\begin{align}\label{ce}
\|J^s\left[(\mathbf u \cdot \nabla)\mathbf \theta\right] - (\mathbf u \cdot \nabla)(J^s\theta)\|_{L^2}&\leq c \|\nabla\mathbf u\|_{L^{\infty}} \|\theta\|_{H^s} + \|\mathbf u\|_{H^s}\|\nabla\mathbf \theta\|_{L^{\infty}}\\
&\leq c \|\nabla\mathbf u\|_{H^s} \|\theta\|_{H^s} + \|\mathbf u\|_{H^s}\|\nabla\mathbf \theta\|_{H^s}.\nonumber
\end{align}
Fefferman et. al. obtained a refined version of the above estimate in \cite{Fe}, which will be useful in our context (e.g. when no bound is available for $\|\nabla\mathbf \theta\|_{H^s}$).
\begin{lemma}\label{lce}
Given s $>$ $\frac{n}{2}$, then there is a constant $c = c(n, s)$ such that, for all $\mathbf u$, $\mathbf \theta$ with $(\mathbf u$, $\mathbf \theta)$ $\in$ $\H^s(\mathbb R^n) \times H^{s}(\mathbb R^{n})$,
\begin{equation}
\| \Lambda^s\left[(\mathbf u \cdot \nabla)\mathbf \theta\right] - (\mathbf u \cdot \nabla)(\Lambda^s\theta)\|_{L^2}\leq c \|\nabla\mathbf u\|_{\H^s}\|\theta\|_{H^s}.
\end{equation}
\end{lemma}
\noindent See Theorem 1.2  in \cite{Fe} for proof.

\subsection{$BMO$ and Besov spaces.}                                                                      
\begin{definition}
\textbf{The space $BMO$}(Bounded Mean Oscillation) is the Banach space of all functions $f \in L^{1}_{loc}(\mathbb R^{n})$ for which
\[\|f\|_{BMO} = \sup_{Q} \left( \frac{1}{|Q|} \int_{Q} \left| f(x) - f_{Q}\right| \,dx \right) < \infty,\]
where the sup ranges over all cubes $Q \subset \mathbb R^{n},$ and $ f_{Q}$ is the mean of f over $Q$.
\end{definition}
\noindent For more details see \cite{Fn}.

Two distinct advantageous properties BMO has compared to $L^{\infty}$ are that the Riesz transforms are bounded in $BMO$ and the singular integral operators of the Calderon-Zygmund type are also bounded in $BMO$. Hence, one can show that $\|\nabla \mathbf u\|_{BMO} \leq C \|\nabla \times \mathbf u\|_{BMO}$ (see \cite{KT}).

\begin{definition}
$\textbf{The Besov space.}$ Let $s \in \mathbb R$, $1 \leq p, q < \infty.$ The Inhomogeneous Besov space $B^{s}_{p, q}$ are defined as the space of all tempered distributions $f \in S'(\mathbb R^{n})$ such that
\[B^{s}_{p, q} = \left\{ f \in S'(\mathbb R^{n}) : \|f\|_{B^{s}_{p, q}} < \infty \right\},\]
where
\[ \|f\|_{B^{s}_{p, q}} = \left( \sum_{j \geq -1} 2^{jqs} \|\Delta_{j} f\|^{q}_{L^{p}}\right)^{\frac{1}{q}},\]
where $\Delta_{j}$ are the inhomogeneous Littlewood-Paley operators.
\end{definition}

\noindent For more details see Chapter 3 of \cite{LR},  appendix of \cite{YX}.

We also note here the following known fact
\[ \|f\|_{B^{s}_{2, 2}} \approx  \|f\|_{H^s}.\]

It is well known that the Sobolev space $W^{s, p}$ is embedded continuously into $L^{\infty}$ for $sp>n$.  However this embedding is false in the space $W^{k, r}$ when $kr=n$. Brezis-Gallouet \cite{BG} and Brezis-Wainger \cite{BW} provided the following inequality which relates the function spaces $L^{\infty}$ and $W^{s, p}$ at the critical value and was used to prove the existence of global solutions to the nonlinear Schr\"{o}dinger equations.
\begin{lemma}
Let $sp >n$. Then 
\[\|f\|_{L^{\infty}} \leq C \left( 1 + \log^{\frac{r-1}{r}} (1+\|f\|_{W^{s, p}})\right),\]
provided $\|f\|_{W^{k, r}} \leq 1$ for $kr=n$.
\end{lemma}
Similar embedding was investigated by Beale-Kato-Majda \cite{BKM} for vector functions to obtain the blow-up criterion of the solutions to the Euler equations.
\begin{lemma}\label{lsi}
Let $s > \frac{n}{p} + 1,$ then we have
\[\|\nabla f\|_{L^{\infty}} \leq C \left( 1 + \|\nabla \cdot f\|_{L^{\infty}} + \|\nabla \times f\|_{L^{\infty}}\left( 1 + log (e + \|f\|_{W^{s, p}})\right)\right),\]
 for all $f \in W^{s, p}(\mathbb R^{n}).$
\end{lemma}
Kozono and Taniuchi improved the above logarithmic Sobolev inequality in $BMO$ space, and applied the result to the three-dimensional Euler equations to prove that $BMO$-norm of the vorticity controls breakdown of smooth solutions.  
\begin{lemma}\label{kte}
Let $1 < p < \infty$ and let $s > \frac{n}{p}$, then we have
\[\|f\|_{L^{\infty}} \leq C \left(1 +\| f\|_{BMO}(1 + \log^{+} \|f\|_{W^{s,p}})\right),\]
 for all $f \in W^{s, p},$
 where $\log^{+} a = \log a$ if $a\geq 1$ and zero otherwise.
\end{lemma}
\noindent For proof see Theorem 1 of \cite{KT}.

\subsection{Other Useful Inequalities}
We state here the inequalities which have been used frequently in this paper.
\begin{lemma}\label{Ye}
\textbf{Young's Inequality.} For $a, b \in \mathbb R$ and $\varepsilon > 0,$ we have, $ab \leq \frac{a^2}{\varepsilon} + \frac{\varepsilon}{4}b^2.$
\end{lemma}

\begin{lemma}
\textbf{Bihari's Inequality.}\label{BE} Let $y$ and $f$ be non-negative functions defined on $[0, \infty),$ and let $w$ be a continuous non-decreasing function defined on $[0, \infty)$ and $w(u) > 0$ on $(0, \infty).$ If $y$ satisfies the following integral inequality,
\[y(t) \leq \alpha + \int_{0}^{t} f(s) w(y(s)) \,ds, \quad t \in [0, \infty), \]
where $\alpha$ is a non-negative constant, then
\[y(t) \leq G^{-1} \left( G(\alpha) + \int_{0}^{t} f(s) \,ds\right),  \quad t \in [0, T],\]
where the function G is defined by
\[G(x) =  \int_{x_{0}}^{x} \frac{1}{w(t)} \, dt, \quad x \geq 0, x_0 > 0,\]
and $ G^{-1}$ is the inverse function of G and T is chosen so that 
\[ G(\alpha) + \int_{0}^{t} f(s) \,ds \in Dom( G^{-1}), \quad \forall \left. t \in [0, T].\right.\]
\end{lemma}
\noindent For proof, see Lemma 1  of \cite{DD}.
\begin{remark} 
Throughout the following sections, $C$ denotes a generic constant.
\end{remark}

\section{Boussinesq Equations with Zero Thermal Diffusion ($B_{\nu, 0}$)}
The viscous, zero thermal diffusive Boussinesq system $B_{\nu, 0}$ is given by: 
\begin{align}\label{B1}
\frac{\partial \mathbf u}{\partial t} + (\mathbf u \cdot \nabla)\mathbf u - \nu \Delta \mathbf u + \nabla p = \mathbf \theta f,  \ \ \textrm{in}\ \mathbb{R}^n\times (0, \infty),
\end{align}
\begin{align}\label{B2}
\frac{\partial \mathbf \theta}{\partial t} + (\mathbf u \cdot \nabla)\mathbf \theta = 0,  \ \ \textrm{in}\ \mathbb{R}^n\times (0, \infty), 
\end{align}
\begin{align}\label{B3}
 \nabla \cdot \mathbf u = 0,  \ \ \textrm{in}\ \mathbb{R}^n\times (0, \infty),
\end{align}
\begin{align}\label{B4}
\mathbf u(x,0) = \mathbf u_{0}(x) , \mathbf \theta(x,0) = \mathbf \theta_{0}(x),  \ \ \textrm{in}\ \mathbb{R}^n,
\end{align}
for $n$ = $2, 3$, with $\nabla\cdot\mathbf u_0 = 0$, and $\mathbf u_{0}\in\H^s(\mathbf{R}^n), \mathbf \theta_{0}\in H^s(\mathbf{R}^n), \ s>n/2$. 
Applying the Fourier truncation operator  $\mathcal S_{\textit{R}}$ (defined in subsection \ref{FTO}) on the above equations, we obtain the truncated Boussinesq system on the whole of $\mathbb R^n$ as:
\begin{align}\label{tr1}
\frac{\partial \mathbf u^{\textit R}}{\partial t} + {\mathcal S}_{\textit R} \left[ (\mathbf u^{\textit R} \cdot \nabla)\mathbf u^{\textit R} \right] - \nu \Delta \mathbf u^{\textit R} + \nabla p^{\textit R} = \theta^{\textit R} f^{\textit R},
\end{align}
\begin{align}\label{tr2}
\frac{\partial \mathbf \theta^{\textit R}}{\partial t} + {\mathcal S}_{\textit R} \left[ (\mathbf u^{\textit R} \cdot \nabla)\mathbf \theta^{\textit R} \right] = 0 ,
\end{align}
\begin{align*}
\nabla \cdot \mathbf u^{\textit R} = 0,
\end{align*}
\begin{align*}
\mathbf u^{\textit R}(0) = {\mathcal S}_{\textit R} \mathbf u_{0}, \mathbf \theta^{\textit R}(0) = {\mathcal S}_{\textit R} \mathbf \theta_{0}.
\end{align*}

By taking the truncated initial data we ensure that $\mathbf u^{\textit R}$, $\mathbf \theta^{\textit R}$ lie in the space 
\begin{align*}
V_R^{\sigma} := \left\{ g \in L^2(\mathbb R^n) : supp(\widehat g) \subset B_R, \nabla \cdot g = 0 \right\}
\end{align*}
and
\[  V_R := \left\{ g \in L^2(\mathbb R^n) : supp(\widehat g) \subset B_R \right\} \] respectively as the truncations are invariant under the flow of the equation \cite{Fe}.\\
\noindent
The divergence free condition for $\mathbf u^{\textit R}$ can be obtained easily as
\[ \widehat {\nabla \cdot \mathbf u^{\textit R}}(\xi) = i\xi \cdot \mathbf 1_{B_R}(\xi) \widehat {\mathbf u} (\xi) = \mathbf 1_{B_R}(\xi) i\xi \cdot  \widehat {\mathbf u} (\xi) = \mathbf 1_{B_R}(\xi) \widehat {\nabla \cdot \mathbf u} (\xi)  = 0  \]

\begin{proposition}\label{cut}
Let $(\mathbf u^{\textit R}$, $\mathbf \theta^{\textit R})$ $\in$ $\H^s(\mathbb R^n) \times H^s(\mathbb R^n)$, for $s > n/2$. Then the nonlinear operator $F(\mathbf u^{\textit R}, \mathbf \theta^{\textit R}) := {\mathcal S}_{\textit R} \left[ (\mathbf u^{\textit R} \cdot \nabla)\mathbf \theta^{\textit R} \right]$ is Lipschitz in $\mathbf u^{\textit R}$ and $\mathbf \theta^{\textit R}$ on the space $V_R^{\sigma} \times V_R$.
\end{proposition}

\begin{proof} Let  $\mathbf \theta^{\textit R}$ $\in$ $H^s(\mathbb R^n)$, for $s > n/2$. Then for proving $F(\cdot, \cdot)$ to be locally Lipschitz in $\mathbf u^{\textit R}$, we use integration by parts, H\"{o}lder's inequality and Sobolev inequality to the term $ \left(F(\mathbf u_1^{\textit R}, \mathbf \theta^{\textit R}) - F(\mathbf u_2^{\textit R}, \mathbf \theta^{\textit R}), \mathbf u_1^{\textit R} - \mathbf u_2^{\textit R} \right)_{L^2}$ to get,
\begin{align*}
\left| \left(F(\mathbf u_1^{\textit R}, \mathbf \theta^{\textit R}) - F(\mathbf u_2^{\textit R}, \mathbf \theta^{\textit R}), \mathbf u_1^{\textit R} - \mathbf u_2^{\textit R} \right)_{L^2} \right|  &=  \left| \left({\mathcal S}_{\textit R} \left[ (\mathbf u_1^{\textit R} - \mathbf u_2^{\textit R})  \cdot \nabla \mathbf \theta^{\textit R} \right], \mathbf u_1^{\textit R} - \mathbf u_2^{\textit R} \right)_{L^2} \right|
\\ &= \left| \left( \left( (\mathbf u_1^{\textit R} - \mathbf u_2^{\textit R})  \cdot \nabla  \right)\mathbf \theta^{\textit R}, {\mathcal S}_{\textit R}( \mathbf u_1^{\textit R} - \mathbf u_2^{\textit R}) \right)_{L^2} \right|
\\ &= \left| - \left( \left( (\mathbf u_1^{\textit R} - \mathbf u_2^{\textit R})  \cdot \nabla  \right)( \mathbf u_1^{\textit R} - \mathbf u_2^{\textit R}), {\mathcal S}_{\textit R}\mathbf \theta^{\textit R} \right)_{L^2} \right|
\\ &\leq \|\mathbf u_1^{\textit R} - \mathbf u_2^{\textit R}\|_{\L^2} \| \nabla(\mathbf u_1^{\textit R} - \mathbf u_2^{\textit R})\|_{\L^2} \|{\mathcal S}_{\textit R}\mathbf \theta^{\textit R} \|_{L^{\infty}}
\\ &\leq \|\mathbf u_1^{\textit R} - \mathbf u_2^{\textit R}\|_{\L^2} \| \mathbf u_1^{\textit R} - \mathbf u_2^{\textit R} \|_{\H^1} \|\mathbf \theta^{\textit R} \|_{L^{\infty}}
\\ &\leq C \|\mathbf u_1^{\textit R} - \mathbf u_2^{\textit R}\|_{\H^s} \|\mathbf \theta^{\textit R} \|_{H^s} \|\mathbf u_1^{\textit R} - \mathbf u_2^{\textit R}\|_{\L^2}.
\end{align*}
Therefore, $\| F(\mathbf u_1^{\textit R}, \mathbf \theta^{\textit R}) - F(\mathbf u_2^{\textit R}, \mathbf \theta^{\textit R})\|_{L^2} \leq C \| \mathbf \theta^{\textit R} \|_{H^s} \|\mathbf u_1^{\textit R} - \mathbf u_2^{\textit R}\|_{\H^s}$, and hence $F(\cdot, \cdot)$ is locally Lipschitz in $\mathbf u^{\textit R}$ on $V_R^{\sigma}$. To prove that $F$ is locally Lipschitz in $\mathbf \theta^{\textit R}$ on $V_R$, we use Remark \ref{div}. For $s > n/2$ and $\mathbf u^{\textit R} \in \H^s(\mathbb R^{n}),$ we have
\begin{align*}
\left| \left(F(\mathbf u^{\textit R}, \mathbf \theta_1^{\textit R}) - F(\mathbf u^{\textit R}, \mathbf \theta_2^{\textit R}), \mathbf \theta_1^{\textit R} - \mathbf \theta_2^{\textit R} \right)_{L^2} \right| &= \left| \left({\mathcal S}_{\textit R} (\mathbf u^{\textit R} \cdot \nabla) (\mathbf \theta_1^{\textit R}-\mathbf \theta_2^{\textit R}), \mathbf \theta_1^{\textit R} - \mathbf \theta_2^{\textit R} \right)_{L^2} \right|
\\ &= \left| \left( (\mathbf u^{\textit R} \cdot \nabla) (\mathbf \theta_1^{\textit R}-\mathbf \theta_2^{\textit R}), {\mathcal S}_{\textit R}(\mathbf \theta_1^{\textit R} - \mathbf \theta_2^{\textit R}) \right)_{L^2} \right|
\\ &\leq \| (\mathbf u^{\textit R} \cdot \nabla) (\mathbf \theta_1^{\textit R}-\mathbf \theta_2^{\textit R}) \|_{L^2} \|{\mathcal S}_{\textit R}(\mathbf \theta_1^{\textit R} - \mathbf \theta_2^{\textit R}) \|_{L^2}
\\ &\leq C \| \mathbf u^{\textit R}\|_{\H^1} \| \nabla (\mathbf \theta_1^{\textit R}-\mathbf \theta_2^{\textit R}) \|_{H^{s-1}} \| \mathbf \theta_1^{\textit R} - \mathbf \theta_2^{\textit R} \|_{L^2}
\\ &\leq C \| \mathbf u^{\textit R}\|_{\H^s} \| \mathbf \theta_1^{\textit R}-\mathbf \theta_2^{\textit R} \|_{H^{s}} \| \mathbf \theta_1^{\textit R} - \mathbf \theta_2^{\textit R} \|_{L^2}.
\end{align*}
Therefore, $\| \left(F(\mathbf u^{\textit R}, \mathbf \theta_1^{\textit R}) - F(\mathbf u^{\textit R}, \mathbf \theta_2^{\textit R}\right) \|_{L^2} \leq C \| \mathbf u^{\textit R}\|_{\H^s} \| \mathbf \theta_1^{\textit R}-\mathbf \theta_2^{\textit R} \|_{H^{s}}$
and hence $F(\cdot, \cdot)$ is locally Lipschitz in $\mathbf \theta^{\textit R}$ on $V_R$.
\end{proof}

Also by using Plancherel's theorem, we get $\Delta \mathbf u^{\textit R}$ has a bounded linear growth (depending on $R$) in $V_R^\sigma$, since
\begin{align*}
\|\Delta \mathbf u^{\textit R}\|^2_{\L^2(\mathbb R^n)} = \| \Delta \widehat {\mathbf u^{\textit R}}\|^2_{\L^2(\mathbb R^n)} = \| \Delta \widehat {\mathbf u^{\textit R}}\|^2_{V_R^\sigma} = \| | \xi |^2 \widehat {\mathbf u^{\textit R}}\|^2_{V_R^\sigma}  &\leq R^2 \| \widehat {\mathbf u^{\textit R}}\|^2_{\L^2(\mathbb R^n)}  
\\ &= R^2 \| \mathbf {u^{\textit R}}\|^2_{\L^2(\mathbb R^n)}.
\end{align*}

Hence by Picard's theorem for Hilbert or Banach space-valued ordinary differential equations, there exists a solution $(\mathbf u^{\textit R}, \mathbf \theta^{\textit R})$ in $V_R^{\sigma}\times V^R$ for some interval $[0, T]$. The solution will exist as long as $\| \mathbf u^{\textit R}\|_{\H^s}$ , $ \| \mathbf \theta^{\textit R}\|_{H^s}$ remain finite.

\subsection{Energy Estimates.}\label{EM}

In this subsection we will establish $H^s$-energy estimates for $\mathbf u^{\textit R}$ and $\mathbf \theta^{\textit R}$, and show that $\| \mathbf u^{\textit R}\|_{\H^s}$ , $ \| \mathbf \theta^{\textit R}\|_{H^s}$  are bounded up to a local time independently of $R$.

\begin{proposition}\label{fin}
Let $(\mathbf u_{0}, \mathbf \theta_{0})$ $\in$ $\H^s(\mathbb R^n) \times H^s(\mathbb R^n)$ with $s > n/2$. Then there exists a time $\tilde{T} = \tilde{T}(s, \nu, \|u_0\|_{\H^{s}}, \|\theta_0\|_{H^{s}}) > 0$ such that the following norms
\[ \sup_{t \in [0, \tilde T]} \| \mathbf u^{\textit R}(t)\|_{\H^s}, \left. \left. \left. \right. \right. \right. \sup_{t \in [0, \tilde T]}\| \mathbf \theta^{\textit R}(t)\|_{H^s}, \left. \left. \left. \right. \right. \right. \int_{0}^{\tilde T} \| \nabla \mathbf u^{\textit R}(t)\|^2_{\H^s} \,dt\]
are bounded uniformly in $R$.
\end{proposition}

\begin{proof}
Let $\Lambda^s$ denote the fractional derivative operator as defined earlier.

Now for $s > n/2$, apply $\Lambda^s$ to both the equations \eqref{tr1} and \eqref{tr2}:
\begin{align}\label{ltr1}
\frac{\partial \Lambda^s \mathbf u^{\textit R}}{\partial t} + {\mathcal S}_{\textit R} \Lambda^s \left[ (\mathbf u^{\textit R} \cdot \nabla)\mathbf u^{\textit R} \right] - \nu \Delta \Lambda^s \mathbf u^{\textit R} + \nabla \Lambda^s p^{\textit R} = \Lambda^s \theta^{\textit R} f^{\textit R},
\end{align}
\begin{align}\label{ltr2}
\frac{\partial \Lambda^s \mathbf \theta^{\textit R}}{\partial t} + {\mathcal S}_{\textit R} \Lambda^s \left[ (\mathbf u^{\textit R} \cdot \nabla)\mathbf \theta^{\textit R} \right] = 0.
\end{align}

Taking ${L^2}$-inner product of \eqref{ltr1} with $\Lambda^s \mathbf u^{\textit R}$ and ${L^2}$-inner product of \eqref{ltr2} with $\Lambda^s \mathbf \theta^{\textit R}$ :
\begin{align}\label{sltr1}
\left( \frac{\partial \Lambda^s \mathbf u^{\textit R}}{\partial t}, \Lambda^s \mathbf u^{\textit R} \right)_{L^2} &+ \left( {\mathcal S}_{\textit R} \Lambda^s \left[ (\mathbf u^{\textit R} \cdot \nabla)\mathbf u^{\textit R} \right], \Lambda^s \mathbf u^{\textit R} \right)_{L^2} - \left( \nu \Delta \Lambda^s \mathbf u^{\textit R}, \Lambda^s \mathbf u^{\textit R} \right)_{L^2} \nonumber
\\ &+ \left( \nabla \Lambda^s p^{\textit R}, \Lambda^s \mathbf u^{\textit R} \right)_{L^2} = \left( \Lambda^s \theta^{\textit R} f^{\textit R}, \Lambda^s \mathbf u^{\textit R} \right)_{L^2},
\end{align}
\begin{align}\label{sltr2}
\left( \frac{\partial \Lambda^s \mathbf \theta^{\textit R}}{\partial t}, \Lambda^s \mathbf \theta^{\textit R} \right)_{L^2} + \left( {\mathcal S}_{\textit R} \Lambda^s \left[ (\mathbf u^{\textit R} \cdot \nabla)\mathbf \theta^{\textit R} \right], \Lambda^s \mathbf \theta^{\textit R} \right)_{L^2} = 0.
\end{align}

We now estimate each term of \eqref{sltr1} and \eqref{sltr2}.

\begin{enumerate}
\item
$\left( \frac{\partial \Lambda^s \mathbf u^{\textit R}}{\partial t}, \Lambda^s \mathbf u^{\textit R} \right)_{L^2}$
\[= \int_{B_R} \frac{\partial \Lambda^s \mathbf u^{\textit R}}{\partial t} \Lambda^s \mathbf u^{\textit R} \,dx = \frac{1}{2} \int_{B_R} \frac{\partial \left| \Lambda^s \mathbf u^{\textit R} \right|^2}{\partial t} = \frac{1}{2} \frac{d}{dt} \|\Lambda^s \mathbf u^{\textit R}\|^2_{\L^2}.\]

\item
Applying weak Parseval's identity and then using the fact that ${\mathcal S}_{\textit R} \mathbf u^{\textit R} = \mathbf u^{\textit R}$, since $\mathbf u^{\textit R} \in V^{\sigma}_R$ we get,
\[ \left( {\mathcal S}_{\textit R} \Lambda^s \left[ (\mathbf u^{\textit R} \cdot \nabla)\mathbf u^{\textit R} \right], \Lambda^s \mathbf u^{\textit R} \right)_{L^2} = \left( \Lambda^s \left[ (\mathbf u^{\textit R} \cdot \nabla)\mathbf u^{\textit R} \right], \Lambda^s \mathbf u^{\textit R} \right)_{L^2}. \]
By H\"older's inequality and Remark \ref{prodhs}, we have,
\begin{align*}
\left| \left( \Lambda^s \left[ (\mathbf u^{\textit R} \cdot \nabla)\mathbf u^{\textit R} \right], \Lambda^s \mathbf u^{\textit R} \right)_{L^2}\right| &\leq C \|\Lambda^s \left[ (\mathbf u^{\textit R} \cdot \nabla)\mathbf u^{\textit R} \right]\|_{\L^2} \|\Lambda^s \mathbf u^{\textit R}\|_{\L^2}
\\ &\leq C \|(\mathbf u^{\textit R} \cdot \nabla)\mathbf u^{\textit R}\|_{\dH^s} \| \mathbf u^{\textit R}\|_{\dH^s}
\\ &\leq C \|(\mathbf u^{\textit R} \cdot \nabla)\mathbf u^{\textit R}\|_{\H^s} \| \mathbf u^{\textit R}\|_{\H^s}
\\ &\leq C \|\mathbf u^{\textit R}\|_{\H^s} \| \nabla \mathbf u^{\textit R}\|_{\H^s} \| \mathbf u^{\textit R}\|_{\H^s}
\\ &= C \|\nabla \mathbf u^{\textit R}\|_{\H^s} \| \mathbf u^{\textit R}\|^2_{\H^s}.
\end{align*}

\item
Integration by parts yields,
\begin{align*}
- \left( \nu \Delta \Lambda^s \mathbf u^{\textit R}, \Lambda^s \mathbf u^{\textit R} \right)_{L^2} &= - \nu \int_{B^R} \Delta \Lambda^s \mathbf u^{\textit R}\cdot \Lambda^s \mathbf u^{\textit R} \,dx ,
\\ &= \nu \int_{B^R} \nabla \Lambda^s \mathbf u^{\textit R}\cdot \nabla \Lambda^s \mathbf u^{\textit R} \,dx = \nu \| \Lambda^s \nabla \mathbf u^{\textit R}\|^2_{\L^2}.
\end{align*}

\item
Again integration by parts and incompressibility condition yield $$\left( \nabla \Lambda^s p^{\textit R}, \Lambda^s \mathbf u^{\textit R} \right)_{L^2}=  -\left( \Lambda^s p^{\textit R}, \Lambda^s \nabla \cdot \mathbf u^{\textit R} \right)_{L^2} = 0.$$

\item Since $f\in L^2 \left(0, T ; H^s(\mathbb R^n)\right)$, 
\begin{align*}
\left| \left( \Lambda^s \theta^{\textit R} f^{\textit R}, \Lambda^s \mathbf u^{\textit R} \right)_{L^2} \right| &\leq \| \Lambda^s \theta^{\textit R} f^{\textit R} \|_{L^2} \| \Lambda^s \mathbf u^{\textit R} \|_{\L^2}
\\ &\leq C \| \theta^{\textit R} f^{\textit R} \|_{\dot H^s} \| \mathbf u^{\textit R} \|_{\dH^s}
\\ &\leq C \| \theta^{\textit R} f^{\textit R} \|_{H^s} \| \mathbf u^{\textit R} \|_{\H^s}
\\ &\leq C \| f^{\textit R} \|_{H^s} \| \theta^{\textit R}\|_{H^s} \| \mathbf u^{\textit R} \|_{\H^s}.
\end{align*}

\item
Similarly,
 \[ \left( \frac{\partial \Lambda^s \mathbf \theta^{\textit R}}{\partial t}, \Lambda^s \mathbf \theta^{\textit R} \right)_{L^2} =  \frac{1}{2} \frac{d}{dt} \|\Lambda^s \mathbf \theta^{\textit R}\|^2_{L^2}.\]
\item
Finally we observe that as in the calculation of the 2nd term,
\[ \left( {\mathcal S}_{\textit R} \Lambda^s \left[ (\mathbf u^{\textit R} \cdot \nabla)\mathbf \theta^{\textit R} \right], \Lambda^s \mathbf \theta^{\textit R} \right)_{L^2} = \left( \Lambda^s \left[ (\mathbf u^{\textit R} \cdot \nabla)\mathbf \theta^{\textit R} \right], \Lambda^s \mathbf \theta^{\textit R} \right)_{L^2}.\]
Taking $L^2$-inner product of  $\Lambda^s \left[ (\mathbf u^{\textit R} \cdot \nabla)\mathbf \theta^{\textit R} \right] - (\mathbf u^{\textit R} \cdot \nabla)(\Lambda^s \mathbf \theta^{\textit R})$ with $\Lambda^s \mathbf \theta^{\textit R}$ we obtain,
\begin{align*}
&\left| \left( \Lambda^s \left[ (\mathbf u^{\textit R} \cdot \nabla)\mathbf \theta^{\textit R} \right] - (\mathbf u^{\textit R} \cdot \nabla)(\Lambda^s \mathbf \theta^{\textit R}), \Lambda^s \mathbf \theta^{\textit R}\right)_{L^2} \right| 
\\ &= \left| \left( \Lambda^s \left[ (\mathbf u^{\textit R} \cdot \nabla)\mathbf \theta^{\textit R} \right], \Lambda^s \mathbf \theta^{\textit R}\right)_{L^2} - \left((\mathbf u^{\textit R} \cdot \nabla)(\Lambda^s \mathbf \theta^{\textit R}), \Lambda^s \mathbf \theta^{\textit R}\right)_{L^2} \right|
\\ &= \left| \left( \Lambda^s \left[ (\mathbf u^{\textit R} \cdot \nabla)\mathbf \theta^{\textit R} \right], \Lambda^s \mathbf \theta^{\textit R}\right)_{L^2}\right|.
\end{align*}
Now Lemma \ref{lce} yields,
\begin{align*}
& \left| \left( \Lambda^s \left[ (\mathbf u^{\textit R} \cdot \nabla)\mathbf \theta^{\textit R} \right] - (\mathbf u^{\textit R} \cdot \nabla)(\Lambda^s \mathbf \theta^{\textit R}), \Lambda^s \mathbf \theta^{\textit R}\right)_{L^2} \right|
\\ &\leq \| \Lambda^s \left[ (\mathbf u^{\textit R} \cdot \nabla)\mathbf \theta^{\textit R} \right] - (\mathbf u^{\textit R} \cdot \nabla)(\Lambda^s \mathbf \theta^{\textit R})\|_{L^2} \| \Lambda^s \mathbf \theta^{\textit R}\|_{L^2}
\\ &\leq C \| \nabla \mathbf u^{\textit R}\|_{\H^s} \| \mathbf \theta^{\textit R}\|_{H^s} \| \mathbf \theta^{\textit R}\|_{\dot H^s}
\leq C \| \nabla \mathbf u^{\textit R}\|_{\H^s} \| \mathbf \theta^{\textit R}\|^2_{H^s}
\end{align*}
So we get,
\[ \left| \left( \Lambda^s \left[ (\mathbf u^{\textit R} \cdot \nabla)\mathbf \theta^{\textit R} \right], \Lambda^s \mathbf \theta^{\textit R}\right)_{L^2} \right| \leq C \| \nabla \mathbf u^{\textit R}\|_{\H^s} \| \mathbf \theta^{\textit R}\|^2_{H^s}.\]
\end{enumerate}

Now adding both \eqref{sltr1} and \eqref{sltr2} and using all the estimates in steps (1)-(7) we obtain,
\begin{align}\label{hsdee}
\frac{1}{2} \frac{d}{dt} &\left( \|\mathbf u^{\textit R}\|^2_{\dot H^s}  + \| \mathbf \theta^{\textit R}\|^2_{\dot H^s}\right) +  \nu \| \nabla \mathbf u^{\textit R}\|^2_{\dot H^s} \nonumber
\\ &\leq  C \| \nabla \mathbf u^{\textit R}\|_{\H^s} (\| \mathbf \theta^{\textit R}\|^2_{H^s} +  \| \mathbf u^{\textit R}\|^2_{\H^s}) +  C \| f^{\textit R} \|_{H^s} \| \theta^{\textit R}\|_{H^s} \| \mathbf u^{\textit R} \|_{\H^s}.
\end{align}

Similarly taking ${L^2}$-inner product of \eqref{tr1} with $\mathbf u^{\textit R}$ and \eqref{tr2} with $\mathbf \theta^{\textit R}$, and using the estimate,
\begin{align*}
\left| (\mathbf \theta^{\textit R} f^{\textit R}, \mathbf u^{\textit R})\right| \leq \| \mathbf \theta^{\textit R} f^{\textit R}\|_{L^2} \| \mathbf u^{\textit R} \|_{\L^2} &\leq \| \mathbf \theta^{\textit R}\|_{L^2} \| f^{\textit R}\|_{L^\infty} \| \mathbf u^{\textit R} \|_{\L^2}\\ &\leq  \| f^{\textit R}\|_{H^s} \| \mathbf \theta^{\textit R}\|_{L^2} \| \mathbf u^{\textit R} \|_{\L^2}.
\end{align*}
we obtain the $L^2$-energy estimate as,
\begin{align}\label{l2ee}
\frac{1}{2} \frac{d}{dt} \left( \|\mathbf u^{\textit R}\|^2_{\L^2}  + \| \mathbf \theta^{\textit R}\|^2_{L^2}\right) +  \nu \| \nabla \mathbf u^{\textit R}\|^2_{\L^2} \leq  C \| f^{\textit R} \|_{H^s} \| \theta^{\textit R}\|_{H^s} \| \mathbf u^{\textit R} \|_{\H^s}
\end{align}
Now adding equations \eqref{hsdee} and \eqref{l2ee}, and applying Young's inequality we obtain,
\begin{align}\label{hse}
\frac{d}{dt} &\left( \|\mathbf u^{\textit R}\|^2_{\H^s}  + \| \mathbf \theta^{\textit R}\|^2_{H^s}\right) +  2\nu \| \nabla \mathbf u^{\textit R}\|^2_{\H^s} \nonumber
\\ &\leq  2C \| \nabla \mathbf u^{\textit R}\|_{\H^s} (\| \mathbf u^{\textit R}\|^2_{\H^s} +  \| \mathbf \theta^{\textit R}\|^2_{H^s}) +  2C \| f^{\textit R} \|_{H^s} \| \mathbf u^{\textit R}\|_{\H^s} \| \mathbf \theta^{\textit R} \|_{H^s}\nonumber
\\ &\leq \frac{C_1}{\nu} (\| \mathbf u^{\textit R}\|^2_{\H^s} +  \| \mathbf \theta^{\textit R}\|^2_{H^s})^2 + \nu \| \nabla \mathbf u^{\textit R}\|^2_{\H^s} + C_2 \| f^{\textit R} \|_{H^s} \left( \| \mathbf u^{\textit R}\|^2_{\H^s} + \| \mathbf \theta^{\textit R} \|^2_{H^s} \right).
\end{align}
Rearranging we get,
\begin{align}
\frac{d}{dt} &\left( \|\mathbf u^{\textit R}\|^2_{\H^s}  + \| \mathbf \theta^{\textit R}\|^2_{H^s}\right) + \nu \| \nabla \mathbf u^{\textit R}\|^2_{\H^s} \nonumber
\\ &\leq \frac{C_1}{\nu} (\| \mathbf u^{\textit R}\|^2_{\H^s} +  \| \mathbf \theta^{\textit R}\|^2_{H^s})^2 + C_2 \| f^{\textit R} \|_{H^s} \left( \| \mathbf u^{\textit R}\|^2_{\H^s} + \| \mathbf \theta^{\textit R} \|^2_{H^s} \right).
\end{align}
Let $X(t) = \|\mathbf u^{\textit R}(t)\|^2_{\H^s}  + \| \mathbf \theta^{\textit R}(t)\|^2_{H^s}.$ Then 
\begin{align*}
\frac{dX(t)}{dt} &\leq \frac{C_1}{\nu} X(t)^2 + C_2 \|f^{\textit R}\|_{H^s} X(t),
\\ &\leq \frac{C_1}{\nu} X(t)^2 + \frac{C_2}{2} \left( \|f^{\textit R}\|^2_{H^s} +  X(t)^2 \right).
\end{align*}
Integrating in $t \in [0, T]$, we have,
\begin{align*}
X(t) &\leq X_{0} + \frac{C_1}{\nu} \int^{t}_{0} X(t)^2 \,dt + \frac{C_2}{2} \left( \int^{t}_{0} \|f^{\textit R}\|^2_{H^s} \,dt +  \int^{t}_{0} X(t)^2 \,dt \right)
\\ &\leq X_{0} + \left( \frac{2 C_1 + \nu C_2}{2 \nu} \right) \int^{T}_{0} X(t)^2 \,dt + C_3 
 \\ &\leq C + \left( \frac{2 C_1 + \nu C_2}{2 \nu} \right) \int^{T}_{0} X(t)^2 \,dt .
\end{align*}
Therefore Bihari's inequality (see Lemma \ref{BE}), yields for all $0 \leq t \leq T$,
\[ X(t) \leq \frac{C}{1 - \left( \frac{2 C_1 + \nu C_2}{2 \nu}\right)CT } = \frac{X_0 + C_3}{1 - \left( \frac{2 C_1 + \nu C_2}{2 \nu}\right)(X_0 + C_3)T }.\]
So provided we choose $\tilde T < \frac{2 \nu}{(2 C_1 + \nu C_2)(X_0 + C_3)}$, $\| \mathbf u^{\textit R}\|_{\H^s}$ and  $\| \mathbf \theta^{\textit R} \|_{H^s}$ remain bounded on $[0, \tilde T]$ independent of $R$.

Finally going back to \eqref{hse} and dropping the first term from the left hand side we get for $0 \leq t \leq \tilde T$,
\[\int^{\tilde T}_{0} \| \nabla \mathbf u^{\textit R}(t)\|^2_{\H^s} \,dt \leq C < \infty.\]
\end{proof}

\subsection{Local Existence and Uniqueness.}\label{EU}

In this subsection, we will prove the existence and uniqueness of the local-in time strong solution Boussinesq equations \eqref{B1} and \eqref{B2}. Throughout we will assume that $f^{\textit R} \in L^{\infty}\left( [0, \tilde T]; H^s(B_{R})\right)$ for every $R>0$. First we will prove the following important result. 

\begin{proposition}\label{Ca}
The family $(\mathbf u^{\textit R}, \mathbf \theta^{\textit R})$ of solutions of the equations are Cauchy $\left( as \left. \right. R \to \infty \right)$ in $L^{\infty}\left([0, \tilde T]; \L^2(\mathbb R^n)\right)\times L^{\infty}\left([0, \tilde T]; L^2(\mathbb R^n)\right)$.
\end{proposition}

\begin{proof}
Consider the equations \eqref{tr1} and \eqref{tr2}, and take the difference between the equations for $R$ and $R'$ where $R' > R$ to get:
\begin{align}\label{dif1}
&\frac{\partial (\mathbf u^{\textit R} - \mathbf u^{\textit R'})}{\partial t}  - \nu \Delta (\mathbf u^{\textit R} - \mathbf u^{\textit R'}) + \nabla (p^{\textit R} - p^{\textit R'}) \nonumber
\\ &= (\theta^{\textit R} f^{\textit R} - \theta^{\textit R'} f^{\textit R'}) - {\mathcal S}_{\textit R} \left[ (\mathbf u^{\textit R} \cdot \nabla)\mathbf u^{\textit R} \right] + {\mathcal S}_{\textit R'} \left[ (\mathbf u^{\textit R'} \cdot \nabla)\mathbf u^{\textit R'} \right],
\end{align}
\begin{align}\label{dif2}
\frac{\partial (\mathbf \theta^{\textit R} - \mathbf \theta^{\textit R'})}{\partial t} + {\mathcal S}_{\textit R} \left[ (\mathbf u^{\textit R} \cdot \nabla)\mathbf \theta^{\textit R} \right] - {\mathcal S}_{\textit R'} \left[ (\mathbf u^{\textit R'} \cdot \nabla)\mathbf \theta^{\textit R'} \right] = 0.
\end{align}
Taking the inner product of \eqref{dif1} with $\mathbf u^{\textit R} - \mathbf u^{\textit R'}$ and the inner product of \eqref{dif2} with $\mathbf \theta^{\textit R} - \mathbf \theta^{\textit R'}$, and calculating in the similar manner as in Proposition \ref{fin}, we obtain
\begin{align}\label{dif}
&\frac{1}{2} \frac{d}{dt}\left( \|\mathbf u^{\textit R} - \mathbf u^{\textit R'}\|^2_{\L^2}  + \| \mathbf \theta^{\textit R} - \mathbf \theta^{\textit R'}\|^2_{L^2}\right) +  \nu \| \nabla (\mathbf u^{\textit R} - \mathbf u^{\textit R'}) \|^2_{\L^2} \nonumber
\\ &= \left( \theta^{\textit R} f^{\textit R} - \theta^{\textit R'} f^{\textit R'}, \mathbf u^{\textit R} - \mathbf u^{\textit R'}\right)\nonumber\\ &\ \ \ - \left( {\mathcal S}_{\textit R} \left[ (\mathbf u^{\textit R} \cdot \nabla)\mathbf u^{\textit R} \right] - {\mathcal S}_{\textit R'} \left[ (\mathbf u^{\textit R'} \cdot \nabla)\mathbf u^{\textit R'} \right], \mathbf u^{\textit R} - \mathbf u^{\textit R'}\right) \nonumber
\\ & \ \ \ - \left( {\mathcal S}_{\textit R} \left[ (\mathbf u^{\textit R} \cdot \nabla)\mathbf \theta^{\textit R} \right] - {\mathcal S}_{\textit R'} \left[ (\mathbf u^{\textit R'} \cdot \nabla)\mathbf \theta^{\textit R'} \right], \mathbf \theta^{\textit R} - \mathbf \theta^{\textit R'} \right).
\end{align}
We split $ \left( {\mathcal S}_{\textit R} \left[ (\mathbf u^{\textit R} \cdot \nabla)\mathbf \theta^{\textit R} \right] - {\mathcal S}_{\textit R'} \left[ (\mathbf u^{\textit R'} \cdot \nabla)\mathbf \theta^{\textit R'} \right], \mathbf \theta^{\textit R} - \mathbf \theta^{\textit R'} \right)$ as sum of three parts:
\begin{align}\label{3t}
&\left( ({\mathcal S}_{\textit R} - {\mathcal S}_{\textit R'}) \left[ (\mathbf u^{\textit R} \cdot \nabla)\mathbf \theta^{\textit R} \right], \mathbf \theta^{\textit R} - \mathbf \theta^{\textit R'} \right)\nonumber
+ \left( {\mathcal S}_{\textit R'} \left[ ((\mathbf u^{\textit R}-\mathbf u^{\textit R'}) \cdot \nabla)\mathbf \theta^{\textit R} \right], \mathbf \theta^{\textit R} - \mathbf \theta^{\textit R'} \right)\nonumber
\\&\ \ \ + \left( {\mathcal S}_{\textit R'} \left[ (\mathbf u^{\textit R'} \cdot \nabla)(\mathbf \theta^{\textit R} - \mathbf \theta^{\textit R'}) \right], \mathbf \theta^{\textit R} - \mathbf \theta^{\textit R'} \right),
\end{align}
and estimate each part separately. Note that the third term of \eqref{3t} vanishes, due to weak Parseval's identity, integration by parts, and the divergence free condition of $\mathbf u^{\textit R}$.

For the first term of \eqref{3t}, using the properties of Fourier truncation operator, we obtain,
\[ \| ({\mathcal S}_{\textit R} - {\mathcal S}_{\textit R'}) \left[ (\mathbf u^{\textit R} \cdot \nabla)\mathbf \theta^{\textit R} \right]\|_{H^s} \leq \frac{C}{\textit R^{\epsilon}} \| ( \mathbf u^{\textit R} \cdot \nabla)\mathbf \theta^{\textit R} \|_{H^{s+\epsilon}}.\]

Let $0 < \epsilon < s-1$. Then $H^s$ being an algebra for $s > n/2$, by Remark \ref{r23} we obtain,
\begin{align*}
&\left| \left( ({\mathcal S}_{\textit R} - {\mathcal S}_{\textit R'}) \left[ (\mathbf u^{\textit R} \cdot \nabla)\mathbf \theta^{\textit R} \right], \mathbf \theta^{\textit R} - \mathbf \theta^{\textit R'} \right) \right| \\ &\leq \| ({\mathcal S}_{\textit R} - {\mathcal S}_{\textit R'}) \left[ (\mathbf u^{\textit R} \cdot \nabla)\mathbf \theta^{\textit R} \right] \|_{L^2} \| \mathbf \theta^{\textit R} - \mathbf \theta^{\textit R'}\|_{L^2}
\\ &= \| ({\mathcal S}_{\textit R} - {\mathcal S}_{\textit R'}) \left[ (\mathbf u^{\textit R} \cdot \nabla)\mathbf \theta^{\textit R} \right] \|_{H^0} \| \mathbf \theta^{\textit R} - \mathbf \theta^{\textit R'}\|_{L^2}
\leq \frac{C}{\textit R^{\epsilon}} \| ( \mathbf u^{\textit R} \cdot \nabla)\mathbf \theta^{\textit R} \|_{H^{\epsilon}} \| \mathbf \theta^{\textit R} - \mathbf \theta^{\textit R'}\|_{L^2}
\\ &= \frac{C}{\textit R^{\epsilon}} \| \nabla \cdot ( \mathbf u^{\textit R} \mathbf \theta^{\textit R}) \|_{H^{\epsilon}} \| \mathbf \theta^{\textit R} - \mathbf \theta^{\textit R'}\|_{L^2}
\leq \frac{C}{\textit R^{\epsilon}} \| \mathbf u^{\textit R} \mathbf \theta^{\textit R} \|_{H^s} \| \mathbf \theta^{\textit R} - \mathbf \theta^{\textit R'}\|_{L^2}
\\ &\leq \frac{C}{\textit R^{\epsilon}} \| \mathbf u^{\textit R}\|_{\H^s} \| \mathbf \theta^{\textit R} \|_{H^s} \| \mathbf \theta^{\textit R} - \mathbf \theta^{\textit R'}\|_{L^2} 
\leq \frac{C}{\textit R^{\epsilon}} \left( \| \mathbf u^{\textit R}\|^2_{\H^s} + \| \mathbf \theta^{\textit R} \|^2_{H^s}\right) \| \mathbf \theta^{\textit R} - \mathbf \theta^{\textit R'}\|_{L^2}.
\end{align*}
Estimates of the second term of \eqref{3t} in two and three dimensions are different. In either case, using Parseval's identity, H\"older's inequality, Young's inequality and Remark \ref{r23}, we obtain,
\begin{align*}
&\left| \left( {\mathcal S}_{\textit R'} \left[ ((\mathbf u^{\textit R}-\mathbf u^{\textit R'}) \cdot \nabla)\mathbf \theta^{\textit R} \right], \mathbf \theta^{\textit R} - \mathbf \theta^{\textit R'} \right) \right|\\
&\ \ \ = \left| \left( ((\mathbf u^{\textit R}-\mathbf u^{\textit R'}) \cdot \nabla)\mathbf \theta^{\textit R}, {\mathcal S}_{\textit R'} (\mathbf \theta^{\textit R} - \mathbf \theta^{\textit R'}) \right) \right| 
\\ &\ \ \ \leq \| ((\mathbf u^{\textit R}-\mathbf u^{\textit R'}) \cdot \nabla)\mathbf \theta^{\textit R} \|_{L^2} \|{\mathcal S}_{\textit R'} (\mathbf \theta^{\textit R} - \mathbf \theta^{\textit R'})\|_{L^2}
\\ &\ \ \ \leq C \| \mathbf u^{\textit R}-\mathbf u^{\textit R'} \|_{\H^1} \| \nabla \mathbf \theta^{\textit R} \|_{H^{s-1}} \|{\mathcal S}_{\textit R'} (\mathbf \theta^{\textit R} - \mathbf \theta^{\textit R'})\|_{L^2}
\\ &\ \ \ = (\| \mathbf u^{\textit R}-\mathbf u^{\textit R'} \|_{\H^1}) (C\| \mathbf \theta^{\textit R} \|_{H^{s}} \|{\mathcal S}_{\textit R'} (\mathbf \theta^{\textit R} - \mathbf \theta^{\textit R'})\|_{L^2})
\\ &\ \ \ \leq \frac{\nu}{4}\| \mathbf u^{\textit R}-\mathbf u^{\textit R'} \|^2_{\H^1} + \frac{C}{\nu}\| \mathbf \theta^{\textit R} \|^2_{H^{s}} \|\mathbf \theta^{\textit R} - \mathbf \theta^{\textit R'}\|^2_{L^2}.
\end{align*}
Therefore we obtain,
\begin{align}\label{t2}
&\left| \left( {\mathcal S}_{\textit R} \left[ (\mathbf u^{\textit R} \cdot \nabla)\mathbf u^{\textit R} \right] - {\mathcal S}_{\textit R'} \left[ (\mathbf u^{\textit R'} \cdot \nabla)\mathbf u^{\textit R'} \right], \mathbf \theta^{\textit R} - \mathbf \theta^{\textit R'}\right) \right|\nonumber
\\ &\leq \frac{C}{R^\epsilon} \left( \| \mathbf u^{\textit R}\|^2_{\H^s} + \| \mathbf \theta^{\textit R}\|^2_{H^s} \right) \| \mathbf \theta^{\textit R}-\mathbf \theta^{\textit R'}\|_{L^2} + \frac{\nu}{4} \| \mathbf u^{\textit R}-\mathbf u^{\textit R'}\|^2_{\H^1} \nonumber\\
&\ \ \ \ + \frac{C}{\nu} \| \mathbf \theta^{\textit R}\|^2_{H^s} \| \mathbf \theta^{\textit R}-\mathbf \theta^{\textit R'}\|^2_{L^2}.
\end{align}
Similarly, we split the second term on the right hand side of \eqref{dif}, i.e. the term $\left( {\mathcal S}_{\textit R} \left[ (\mathbf u^{\textit R} \cdot \nabla)\mathbf u^{\textit R} \right] - {\mathcal S}_{\textit R'} \left[ (\mathbf u^{\textit R'} \cdot \nabla)\mathbf u^{\textit R'} \right], \mathbf u^{\textit R} - \mathbf u^{\textit R'}\right)$ as sum of three parts
\begin{align}\label{sur}
&\left( ({\mathcal S}_{\textit R}-{\mathcal S}_{\textit R'}) \left[ (\mathbf u^{\textit R} \cdot \nabla)\mathbf u^{\textit R} \right], \mathbf u^{\textit R} - \mathbf u^{\textit R'}\right)\nonumber + \left( {\mathcal S}_{\textit R'} \left[ ((\mathbf u^{\textit R}-\mathbf u^{\textit R'}) \cdot \nabla)\mathbf u^{\textit R} \right], \mathbf u^{\textit R} - \mathbf u^{\textit R'}\right)\nonumber
\\ & \ \ \ + \left( {\mathcal S}_{\textit R'} \left[ (\mathbf u^{\textit R'} \cdot \nabla)(\mathbf u^{\textit R} - \mathbf u^{\textit R'}) \right], \mathbf u^{\textit R} - \mathbf u^{\textit R'}\right), 
\end{align}
and estimate each part separately as before to get,
\begin{align}\label{t1}
&\left| \left( {\mathcal S}_{\textit R} \left[ (\mathbf u^{\textit R} \cdot \nabla)\mathbf u^{\textit R} \right] - {\mathcal S}_{\textit R'} \left[ (\mathbf u^{\textit R'} \cdot \nabla)\mathbf u^{\textit R'} \right], \mathbf u^{\textit R} - \mathbf u^{\textit R'}\right) \right|\nonumber
\\ &\leq \frac{C}{R^\epsilon} \| \mathbf u^{\textit R}\|^2_{\H^s} \| \mathbf u^{\textit R}-\mathbf u^{\textit R'}\|_{\L^2} +  \frac{\nu}{4} \| \mathbf u^{\textit R}-\mathbf u^{\textit R'}\|^2_{\H^1} \nonumber\\
&\ \ \ \ + \frac{C}{\nu} \| \mathbf u^{\textit R}\|^2_{\H^s} \| \mathbf u^{\textit R}-\mathbf u^{\textit R'}\|^2_{\L^2}.
\end{align}
Now we estimate the first term of the right hand side of \eqref{dif},
\begin{align}\label{t3}
&\left| \left( \theta^{\textit R} f^{\textit R} - \theta^{\textit R'} f^{\textit R'}, \mathbf u^{\textit R} - \mathbf u^{\textit R'}\right) \right|\nonumber\\&\ \ \ \ \leq \| \theta^{\textit R} f^{\textit R} - \theta^{\textit R'} f^{\textit R'} \|_{L^2} \| \mathbf u^{\textit R} - \mathbf u^{\textit R'} \|_{\L^2}\nonumber\\
&\ \ \ \ \leq \left(\| \theta^{\textit R} f^{\textit R} - \theta^{\textit R} f^{\textit R'} \|_{L^2} + \| \theta^{\textit R} f^{\textit R'} -  \theta^{\textit R'} f^{\textit R'} \|_{L^2}\right)\| \mathbf u^{\textit R} - \mathbf u^{\textit R'} \|_{\L^2}
\nonumber \\ &\ \ \ \ \leq \left(\| \theta^{\textit R}\|_{L^\infty} \| f^{\textit R} -  f^{\textit R'}\|_{L^2} + \| f^{\textit R'}\|_{L^\infty} \|\theta^{\textit R} -  \theta^{\textit R'} \|_{L^2}\right) \|\mathbf u^{\textit R} - \mathbf u^{\textit R'} \|_{\L^2}
\nonumber \\ &\ \ \ \ \leq C\left(\| \theta^{\textit R}\|_{H^s} \| f^{\textit R} -  f^{\textit R'}\|_{L^2} + \| f^{\textit R'}\|_{H^s} \|\theta^{\textit R} -  \theta^{\textit R'} \|_{L^2}\right) \|\mathbf u^{\textit R} - \mathbf u^{\textit R'} \|_{\L^2}\nonumber
\\ &\ \ \ \ \leq C \| \theta^{\textit R}\|_{H^s}\left(  \|\mathbf u^{\textit R} - \mathbf u^{\textit R'} \|_{\L^2}^2 + \| f^{\textit R} -  f^{\textit R'}\|_{L^2}^2\right)\nonumber\\
&\ \ \ \ \ \ +  C\| f^{\textit R'}\|_{H^s} \ \left( \|\mathbf u^{\textit R} - \mathbf u^{\textit R'} \|_{\L^2}^2 + \|\theta^{\textit R} -  \theta^{\textit R'} \|_{L^2}^2 \right).
\end{align}
Using the estimates obtained in \eqref{t2}, \eqref{t1} and \eqref{t3} in \eqref{dif}, and rearranging we obtain,
\begin{align}\label{nule}
&\frac{d}{dt} \left( \|\mathbf u^{\textit R} - \mathbf u^{\textit R'}\|^2_{\L^2}  + \| \mathbf \theta^{\textit R} - \mathbf \theta^{\textit R'}\|^2_{L^2}\right) +  \nu \| \nabla (\mathbf u^{\textit R} - \mathbf u^{\textit R'}) \|^2_{\L^2}\nonumber \\
&\ \ \ \ \leq \frac{C}{R^\epsilon} \left( \| \mathbf u^{\textit R}\|^2_{\H^s} + \| \mathbf \theta^{\textit R}\|^2_{H^s} \right) \left( \|\mathbf u^{\textit R} - \mathbf u^{\textit R'}\|_{\L^2}  + \| \mathbf \theta^{\textit R} - \mathbf \theta^{\textit R'}\|_{L^2}\right) 
\nonumber \\ &\ \ \ \ \quad+ \frac{C}{\nu} \left( \| \mathbf u^{\textit R}\|^2_{\H^s} + \| \mathbf \theta^{\textit R}\|^2_{H^s} \right) \left( \|\mathbf u^{\textit R} - \mathbf u^{\textit R'}\|^2_{\L^2} + \| \mathbf \theta^{\textit R} - \mathbf \theta^{\textit R'}\|^2_{L^2}\right)\nonumber
\\ &\ \ \ \ \quad+ C \| \theta^{\textit R}\|_{H^s}\left(  \|\mathbf u^{\textit R} - \mathbf u^{\textit R'} \|_{\L^2}^2 + \|\theta^{\textit R} -  \theta^{\textit R'} \|_{L^2}^2 + \| f^{\textit R} -  f^{\textit R'}\|_{L^2}^2\right)\nonumber
\\ &\ \ \ \ \quad+  C\| f^{\textit R'}\|_{H^s}  \left( \|\mathbf u^{\textit R} - \mathbf u^{\textit R'} \|_{\L^2}^2 + \|\theta^{\textit R} -  \theta^{\textit R'} \|_{L^2}^2 \right).
\end{align}
Setting $Y(t) = \| \mathbf u^{\textit R}-\mathbf u^{\textit R'}\|_{\L^2} + \| \mathbf \theta^{\textit R} - \mathbf \theta^{\textit R'}\|_{L^2}$, using the bounds \[ \sup_{t \in [0, \tilde T]} \| \mathbf u^{\textit R}(t)\|_{\H^s}\leq M, \left. \left. \left. \right. \right. \right. \sup_{t \in [0, \tilde T]}\| \mathbf \theta^{\textit R}(t)\|_{H^s}\leq M, \left. \left. \left. \right. \right. \right.\] and recalling $f^{\textit R} \in L^{\infty}\left( [0, \tilde T]; H^s(B_{R})\right)$ for every $R>0$, we have from \eqref{nule}
\begin{align}\label{nule1}
&\left. Y\right. \frac{dY}{dt} + \nu \| \nabla (\mathbf u^{\textit R} - \mathbf u^{\textit R'}) \|^2_{\L^2}\nonumber
\\ &\ \ \ \ \leq \frac{d}{dt} \left( \|\mathbf u^{\textit R} - \mathbf u^{\textit R'}\|^2_{\L^2}  + \| \mathbf \theta^{\textit R} - \mathbf \theta^{\textit R'}\|^2_{L^2}\right) + \nu \| \nabla (\mathbf u^{\textit R} - \mathbf u^{\textit R'}) \|^2_{\L^2} \nonumber
\\ &\ \ \ \ \leq CMY^2 + \frac{CM}{R^\epsilon}Y.
\end{align}
Therefore, ignoring the second term of the left hand side, we obtain
\begin{align*}
Y \frac{dY}{dt} \leq MY^2 + \frac{M}{R^\epsilon}Y,
\end{align*}
which further yields
\begin{align*}
\frac{dY}{dt} \leq MY + \frac{M}{R^\epsilon}.
\end{align*}
Finally applying Gronwall's lemma, we find
\begin{align}\label{ytr}
 \sup_{t \in [0, \tilde T]} Y(t) \leq \frac{C(\tilde T, M)}{R^\epsilon} \to 0
\end{align}
as $R \to\infty$ (since $R'>R, R'\to\infty$ as well),
concluding that $(\mathbf u^{\textit R}, \mathbf \theta^{\textit R})$ are Cauchy in $L^{\infty}\left([0, \tilde T]; \L^2(\mathbb R^n)\right) \times L^{\infty}\left([0, \tilde T]; L^2(\mathbb R^n)\right)$ as $R \to \infty$.

\end{proof}

\begin{proposition}\label{La}
For any $s' > n/2$, $\Delta \mathbf u^{\textit R} \to \Delta \mathbf u$ strongly in $L^2 \left([0, \tilde T]; \H^{{s'}-1}(\mathbb R^n)\right)$ as $R \to \infty$.
\end{proposition}
\begin{proof}
From the above Proposition it is clear that  $\mathbf u^{\textit R} \to \mathbf u$ strongly in $L^{\infty}\left([0, \tilde T]; \L^2(\mathbb R^n)\right)$ and $\mathbf\theta^{\textit R} \to \mathbf\theta$ strongly in $L^{\infty}\left([0, \tilde T]; L^2(\mathbb R^n)\right)$ as $R \to \infty$.\\
\noindent
Observe that from \eqref{nule1} 
\[ \nu \| \nabla (\mathbf u^{\textit R} - \mathbf u^{\textit R'}) \|^2_{\L^2} \leq CMY^2 + \frac{CM}{R^\epsilon}Y.\]
From \eqref{ytr}, we also note that $Y$ is bounded by $\frac{C(\tilde T, M)}{R^\epsilon}$. So taking integration from $0$ to $\tilde T$ and then tending $R \to \infty$, we can find that $\nabla \mathbf u^{\textit R}$ is Cauchy in $L^{2}\left([0, \tilde T]; \L^2(\mathbb R^n)\right)$ and so $\nabla \mathbf u^{\textit R} \to \nabla \mathbf u$ in $L^{2}\left([0, \tilde T]; \L^2(\mathbb R^n)\right)$.\\
\noindent
Now using Lemma \ref{iss}, for $s' < s,$
\begin{align*}
\sup_{t \in [0, \tilde T]} \|\mathbf u^{\textit R} - \mathbf u\|_{\H^{s'}} &\leq C \sup_{t \in [0, \tilde T]} \left(\|\mathbf u^{\textit R} - \mathbf u\|_{\L^{2}}^{1-s'/s} \|\mathbf u^{\textit R} - \mathbf u\|_{\H^{s}}^{s'/s}\right)
\\ &\leq C \left( \sup_{t \in [0, \tilde T]} \|\mathbf u^{\textit R} - \mathbf u\|_{\L^{2}}\right)^{1-s'/s} \left( \sup_{t \in [0, \tilde T]}  \|\mathbf u^{\textit R} - \mathbf u\|_{\H^{s}}\right)^{s'/s}.
\end{align*}
From Propositions \ref{fin} and \ref{Ca}, we obtain
\begin{align*}
\sup_{t \in [0, \tilde T]} \|\mathbf u^{\textit R} - \mathbf u\|_{\H^{s'}} \leq M  \left( \sup_{t \in [0, \tilde T]} \|\mathbf u^{\textit R} - \mathbf u\|_{\L^{2}}\right)^{1-s'/s} \to 0 \quad \quad as \quad R \to \infty,
\end{align*}
which implies
\begin{align*}\label{urcu}
\mathbf u^{\textit R} \to \mathbf u\left.\right.\left.\right. \textit{strongly in} \left.\right.\left.\right. L^{\infty}\left([0, \tilde T]; \H^{s'}(\mathbb R^n)\right)\ \textit{for any}\  s'<s.
\end{align*}
Similarly, one can show
\begin{align*}
 &\mathbf \theta^{\textit R} \to  \mathbf \theta \quad \textit{strongly in} \quad L^{\infty}\left([0, \tilde T]; H^{s'}(\mathbb R^n)\right)  \textit{for any}\  s'<s,\\
 &\nabla \mathbf u^{\textit R} \to \nabla \mathbf u \quad  \textit{strongly in} \quad L^2 \left([0, \tilde T]; \H^{s'}(\mathbb R^n)\right)  \textit{for any}\  s'<s,
\end{align*}
and thus we obtain, $\Delta \mathbf u^{\textit R} \to \Delta \mathbf u$ strongly in $L^2 \left([0, \tilde T]; \H^{{s'}-1}(\mathbb R^n)\right)$ for any $s' < s.$
\end{proof} 

\begin{proposition}
For any $s' > n/2$, the non-linear terms ${\mathcal S}_{\textit R} \left[ (\mathbf u^{\textit R} \cdot \nabla)\mathbf u^{\textit R} \right] \to (\mathbf u \cdot \nabla)\mathbf u$ and ${\mathcal S}_{\textit R} \left[ (\mathbf u^{\textit R} \cdot \nabla)\mathbf \theta^{\textit R} \right] \to (\mathbf u \cdot \nabla)\mathbf \theta$ strongly in $L^{\infty}\left([0, \tilde T]; H^{s'-1}(\mathbb R^n)\right)$ as $R \to \infty$.
\end{proposition}
\begin{proof}
Using properties of the Fourier truncation and Remark \ref{div}, we get for any $s' > n/2$,
\begin{align*}
 &\sup_{t \in [0, \tilde T]} \|{\mathcal S}_{\textit R} \left[ (\mathbf u^{\textit R} \cdot \nabla)\mathbf u^{\textit R} \right] - (\mathbf u \cdot \nabla)\mathbf u\|_{\H^{s'-1}}
\\ &\ \ \ \ \leq \sup_{t \in [0, \tilde T]}\left( \|{\mathcal S}_{\textit R} \left[ (\mathbf u^{\textit R} - \mathbf u) \cdot \nabla)\mathbf u^{\textit R} \right] \|_{\H^{s'-1}} + \|{\mathcal S}_{\textit R} \left[ (\mathbf u \cdot \nabla)(\mathbf u^{\textit R} - \mathbf u) \right]\|_{\H^{s'-1}} \right)
\\ &\ \ \ \ \leq \sup_{t \in [0, \tilde T]}\left( C \|\left[ (\mathbf u^{\textit R} - \mathbf u) \cdot \nabla)\mathbf u^{\textit R} \right] \|_{\H^{s'-1}} + C \|\left[ (\mathbf u \cdot \nabla)(\mathbf u^{\textit R} - \mathbf u) \right]\|_{\H^{s'-1}} \right)
\\ &\ \ \ \ \leq \sup_{t \in [0, \tilde T]}\left( C \|\mathbf u^{\textit R} - \mathbf u \|_{\H^{s'}} \| \mathbf u^{\textit R}\|_{\H^{s'}} + C \| \mathbf u \|_{\H^{s'}} \| \mathbf u^{\textit R} - \mathbf u\|_{\H^{s'}}\right)
\end{align*}
Clearly, from the Propositions \ref{fin} and  \ref{Ca}, the right hand side tends to 0 as $R \to \infty$. Similarly, one can prove ${\mathcal S}_{\textit R} \left[ (\mathbf u^{\textit R} \cdot \nabla)\mathbf \theta^{\textit R} \right] \to (\mathbf u \cdot \nabla)\mathbf \theta$ strongly in $L^{\infty}\left([0, \tilde T]; H^{s'-1}(\mathbb R^n)\right)$ as $R \to \infty$
\end{proof}
\noindent
Next we will show the convergences of the time derivatives.
\begin{proposition}\label{delur}
For any $s > n/2$, $\frac{\partial \mathbf u^{\textit R}}{\partial t} \to \frac{\partial \mathbf u}{\partial t}$ and $ \frac{\partial \mathbf \theta^{\textit R}}{\partial t} \to  \frac{\partial \mathbf \theta}{\partial t}$ strongly in  $L^{2}\left([0, \tilde T]; \H^{s-1}(\mathbb R^n)\right)$ and $L^{2}\left([0, \tilde T]; H^{s-1}(\mathbb R^n)\right)$ respectively as $R \to \infty$.
\end{proposition}
\begin{proof}
Taking $H^{s-1}$-norm on both sides of the truncated equations \eqref{tr1} and \eqref{tr2} we get,
\[\left\| \frac{\partial \mathbf u^{\textit R}}{\partial t} \right\|_{\H^{s-1}} \leq \| \theta^{\textit R} f^{\textit R}\|_{H^{s-1}} + \|{\mathcal S}_{\textit R} \left[ (\mathbf u^{\textit R} \cdot \nabla)\mathbf u^{\textit R} \right]\|_{\H^{s-1}} + \nu \| \Delta \mathbf u^{\textit R}\|_{\H^{s-1}},\]
\[ \left\| \frac{\partial \mathbf \theta^{\textit R}}{\partial t}\right\|_{H^{s-1}} = \| {\mathcal S}_{\textit R} \left[ (\mathbf u^{\textit R} \cdot \nabla)\mathbf \theta^{\textit R} \right]\|_{H^{s-1}}.\]
Using properties of the Fourier truncation operator and Remarks \ref{prodhs} and \ref{div}, we have for $s > n/2$,
\begin{align}\label{delut1}
&\left\| \frac{\partial \mathbf u^{\textit R}}{\partial t} \right\|_{\H^{s-1}} + \left\| \frac{\partial \mathbf \theta^{\textit R}}{\partial t}\right\|_{H^{s-1}}\nonumber\\
&\ \ \ \ \leq C \|\mathbf \theta^{\textit R}\|_{H^s}\|f^{\textit R}\|_{H^s} + C \|\mathbf u^{\textit R}\|^2_{\H^s} + C \| \Delta \mathbf u^{\textit R}\|_{\H^{s-1}} + C \|\mathbf u^{\textit R}\|_{\H^s}\|\mathbf \theta^{\textit R}\|_{H^s}
\\ &\ \ \ \ \leq C \sup_{t \in [0, \tilde T]}\|\mathbf \theta^{\textit R}\|_{H^s}\sup_{t \in [0, \tilde T]}\|f^{\textit R}\|_{H^s} + C \sup_{t \in [0, \tilde T]}\|\mathbf u^{\textit R}\|^2_{\H^s}  + C \| \Delta \mathbf u^{\textit R}\|_{\H^{s-1}}\nonumber\\
&\ \ \ \ \quad + C \sup_{t \in [0, \tilde T]}\|\mathbf u^{\textit R}\|_{\H^s}\cdot \sup_{t \in [0, \tilde T]}\|\mathbf \theta^{\textit R}\|_{H^s}\nonumber
\end{align}
Using Proposition $\ref{fin}$ and dropping the second term of left hand side, we obtain
\begin{align*}
\left\| \frac{\partial \mathbf u^{\textit R}}{\partial t} \right\|_{\H^{s-1}} \leq C + C \| \Delta \mathbf u^{\textit R}\|_{\H^{s-1}} 
\end{align*}
Finally squaring both sides, using Young's inequality,  and integrating in $t \in [0, \tilde T]$, we have, after recalling $ \Delta \mathbf u^{\textit R} \in  L^{2}\left([0, \tilde T]; \H^{s-1}(\mathbb R^n)\right)$, 
\[ \int_0^{\tilde T} \left\| \frac{\partial \mathbf u^{\textit R}}{\partial t} \right\|^2_{\H^{s-1}} \leq C(\tilde T) .\]
Using Banach-Alaoglu weak$-^{\ast}$ compactness theorem (see Robinson \cite{Ro}), we can extract a subsequence $R_m \to +\infty$ such that
\begin{align}\label{drm}
\frac{\partial \mathbf u^{\textit R_m}}{\partial t} \to \frac{\partial \mathbf u}{\partial t}\left.\right.\left.\right. \textrm{weakly$-^{\ast}$ in} \left.\right. L^{2}\left([0, \tilde T]; \H^{s-1}(\mathbb R^n)\right).
\end{align}
Similar argument works for $\frac{\partial \mathbf \theta^{\textit R}}{\partial t}$ as well.\\
\noindent
Note that $\|\mathbf u^{\textit R}\|_{\H^s}\|\mathbf \theta^{\textit R}\|_{H^s} \to \|\mathbf u\|_{\H^s}\|\mathbf \theta\|_{H^s}$ holds due to the strong convergences of $(\mathbf u^{\textit R}, \mathbf \theta^{\textit R})$ to $(\mathbf u, \mathbf \theta)$ in $L^{\infty}\left([0, \tilde T]; \L^2(\mathbb R^n)\right) \times L^{\infty}\left([0, \tilde T]; L^2(\mathbb R^n)\right)$. Hence all the terms on the right hand side of $\eqref{delut1}$ converge strongly (from Propositions \ref{Ca} and \ref{La}), 
we conclude that the convergences of the time derivatives are strong.
\end{proof}

\begin{proposition}\label{mp7}
For $s > n/2,$ $(\mathbf u, \mathbf \theta)$ lie in the space $L^{\infty}\left([0, \tilde T]; \H^{s}(\mathbb R^n)\right) \cap L^{2}\left([0, \tilde T]; \H^{s+1}(\mathbb R^n)\right) \times L^{\infty}\left([0, \tilde T]; H^{s}(\mathbb R^n)\right).$
\end{proposition}

\begin{proof}
By Banach-Alaoglu weak$-^{\ast}$ compactness theorem (see \cite{Ro} or \cite{Yo}), the uniform bounds in Proposition $\ref{fin}$ guarantee the existence of a subsequence such that
\begin{align*}
&\mathbf u^{\textit R_m} \to \mathbf {u}\left.\right. \left.\right.weakly-^{\ast} \left.\right. in \left.\right.\left.\right. L^{\infty}\left([0, \tilde T]; \H^{s}(\mathbb R^n)\right),
\\ &\mathbf \theta^{\textit R_m} \to \mathbf \theta\left.\right. \left.\right.weakly-^{\ast} \left.\right. in \left.\right.\left.\right. L^{\infty}\left([0, \tilde T]; H^{s}(\mathbb R^n)\right),
\end{align*}
and
\begin{align*}
\nabla \mathbf u^{\textit R_m} \to \nabla \mathbf u \left.\right.\left.\right.weakly\left.\right. in\left.\right. \left.\right. L^{2}\left([0, \tilde T]; \H^{s}(\mathbb R^n)\right),
\end{align*}
which guarantees that the limit satisfies
\begin{align*}\label{ul2}
 \mathbf u \in  L^{\infty}\left([0, \tilde T]; \H^{s}(\mathbb R^n)\right) \left. \cap \right.  L^{2}\left([0, \tilde T]; \H^{s+1}(\mathbb R^n)\right),
\end{align*}
and
\begin{align*}
\mathbf \theta \in L^{\infty}\left([0, \tilde T]; H^{s}(\mathbb R^n)\right).
\end{align*}
\end{proof}

\begin{proposition}\label{p8}
Let $s > n/2$ and $\mathbf u_{0} \in  \H^{s}(\mathbb R^n)$ and $\mathbf \theta_{0} \in  H^{s}(\mathbb R^n)$. Let the solutions $(\mathbf u, \mathbf\theta)$ of the system $B_{\nu, 0}$ have the regularity 
\[ \mathbf u \in  L^{\infty}\left([0, \tilde T]; \H^{s}(\mathbb R^n)\right) \left. \cap \right.  L^{2}\left([0, \tilde T]; \H^{s+1}(\mathbb R^n)\right), \mathbf \theta \in L^{\infty}\left([0, \tilde T]; H^{s}(\mathbb R^n)\right). \]
 Then the solutions $(\mathbf u, \mathbf \theta)$ of the system $B_{\nu, 0}$ are unique in $[0, \tilde T].$  
\end{proposition}

\begin{proof}
The proof of the uniqueness is very similar to the proof of  Proposition \ref{Ca}. Let $( \mathbf u^{\textit R}, \mathbf \theta^{\textit R})$ and $(\mathbf u^{\textit R'},  \mathbf \theta^{\textit R'})$ be two  solutions to the truncated Boussinesq equations  $\eqref{B1}-\eqref{B2}$ for $R' > R.$ Then from \eqref{ytr}, we have,
\[\sup_{t \in [0, \tilde T]}\left(\| \mathbf u^{\textit R}-\mathbf u^{\textit R'}\|_{\L^2} + \| \mathbf \theta^{\textit R} - \mathbf \theta^{\textit R'}\|_{L^2}\right) \leq \frac{C}{R^\epsilon},\]
Now letting $R \to R'$ then letting $R \to +\infty$ we observe,
\[ \mathbf u^{\textit R} \to \mathbf u^{\textit R'} \left.\right.\left.\right.\left.\right. and\left.\right. \left.\right.\left.\right. \mathbf \theta^{\textit R} \to \mathbf \theta^{\textit R'}.\]
This guarantees that the limits $(\mathbf u, \mathbf \theta)$ are unique.
\end{proof}

Now combining all the above results, we will prove the main theorem on local-in-time existence and uniqueness of  strong solutions for the Boussinesq system $B_{\nu, 0}$.
\begin{theorem}\label{mt1}
Let $s > \frac{n}{2},$ $\mathbf u_{0} \in \H^s(\mathbb R^{n})$ and $\mathbf \theta_{0} \in  H^s(\mathbb R^{n})$. Then there exists a unique strong solution $(\mathbf u, \mathbf \theta) \in C([0, \tilde{T}]; \H^{s}(\mathbb R^{n})) \cap L^{2}([0, \tilde{T}]; \H^{s+1}(\mathbb R^{n})) \times C([0, \tilde{T}]; H^{s}(\mathbb R^{n}))$ to the Boussinesq system $B_{\nu, 0}$.
\end{theorem}

\begin{proof}
 First note that, by Proposition \ref{mp7}, we already have $\mathbf u\in L^{2}([0, \tilde{T}];  \H^{s+1}(\mathbb R^{n}))$. So the only part that is left to prove is $(\mathbf u, \mathbf \theta) \in  C \left([0, \tilde T]; \H^{s}(\mathbb R^n)\right) \times  C \left([0, \tilde T]; H^{s}(\mathbb R^n)\right)$. \\
 \noindent
Since by Propositions \ref{delur} and \ref{mp7},  $\mathbf u \in L^{2}\left([0, \tilde T]; \H^{s+1}(\mathbb R^n)\right)$ and $\frac{\partial \mathbf u}{\partial t} \in  L^{2}\left([0, \tilde T]; \H^{s-1}(\mathbb R^n)\right)$, by standard known results of parabolic partial differential equations (e.g., see Theorem 4, section 5.9 of \cite{Ev}), we have, $\mathbf u \in  C \left([0, \tilde T]; \H^{s}(\mathbb {R}^n)\right)$.

To prove $\mathbf \theta \in  C \left([0, \tilde T]; H^{s}(\mathbb R^n)\right)$, consider $t_1, t_2 \in [0, \tilde T]$ such that $0 \leq t_1 \leq t_2 < \tilde T$. Then,
\[ \|\mathbf \theta(t_2) - \mathbf \theta(t_1)\|_{H^s} \approx \|\mathbf \theta(t_2) - \mathbf \theta(t_1)\|_{B^{s}_{2, 2}} = \left\{\sum_{j \in \mathbb Z} \left( 2^{js}\left\| \Delta_j \mathbf \theta(t_2) - \Delta_j \mathbf \theta(t_1)\right\|_{L^2}\right)^2\right\}^{1/2},\]
where $\Delta_j$ is the non-homogenous Littlewood-Paley operators.

Let $\epsilon > 0$ be arbitrarily small. As $\mathbf \theta \in  L^{\infty} \left([0, \tilde T]; H^{s}(\mathbb R^n)\right)$, there exists an integer $N > 0$ such that
\begin{align}\label{tle}
\left\{\sum_{j \geq N} \left(2^{js}\left\| \Delta_j \mathbf \theta(t_2) - \Delta_j \mathbf \theta(t_1)\right\|_{L^2}\right)^2\right\}^{1/2} < \frac{\epsilon}{2}.
\end{align}
But we have
\begin{align*}
&\left\{\sum_{j \in \mathbb Z} \left( 2^{js}\left\| \Delta_j \mathbf \theta(t_2) - \Delta_j \mathbf \theta(t_1)\right\|_{L^2}\right)^2\right\}^{1/2}
\\ &= \left\{ \left(\sum_{j < N} + \sum_{j \geq N}\right) \left( 2^{js}\left\| \Delta_j \mathbf \theta(t_2) - \Delta_j \mathbf \theta(t_1)\right\|_{L^2}\right)^2\right\}^{1/2}.
\end{align*}
Now for $0 \leq t_1 < t_2 < \tilde T$ we have,
\begin{align*}
\Delta_j \mathbf \theta(t_2) - \Delta_j \mathbf \theta(t_1) = \int_{t_1}^{t_2} \frac{\partial}{\partial \tau} \Delta_j \mathbf \theta(\tau) \, d\tau = \int_{t_1}^{t_2} \Delta_j \left[ -(\mathbf u \cdot \nabla) \mathbf \theta\right](\tau) \, d\tau
\end{align*}
So we get,
\begin{align*}
\sum_{j < N} 2^{2js}\left.\| \Delta_j \mathbf \theta(t_2) - \Delta_j \mathbf \theta(t_1)\|^2_{L^2}\right. &= \sum_{j < N} 2^{2js}\left.\left\| \int_{t_1}^{t_2} \Delta_j \left[ -(\mathbf u \cdot \nabla) \mathbf \theta\right](\tau) \, d\tau \right\|^2_{L^2}\right.
\\ &\leq \sum_{j < N} 2^{2js}\left. \left( \int_{t_1}^{t_2} \| \Delta_j (\mathbf u \cdot \nabla \mathbf \theta)\|_{L^2} \, d\tau \right)^2\right.
\\ &= \sum_{j < N} 2^{2j}\left. \left( \int_{t_1}^{t_2} 2^{j(s-1)}\| \Delta_j (\mathbf u \cdot \nabla \mathbf \theta)\|_{L^2} \, d\tau \right)^2\right.
\\ &\leq  \sum_{j < N} 2^{2j}\left. \int_{t_1}^{t_2}\| (\mathbf u \cdot \nabla) \mathbf \theta\|^2_{H^{s-1}} \, d\tau \right.
\\ &\leq \sum_{j < N} 2^{2j}\left. \int_{t_1}^{t_2}\| (\mathbf u \cdot \nabla) \mathbf \theta\|^2_{L^{\infty}\left([0, \tilde T]; H^{s-1}\right)} \, d\tau = I\right.
\end{align*}
As $(\mathbf u, \mathbf \theta) \in L^{\infty}\left([0, \tilde T]; \H^{s}\right) \times L^{\infty}\left([0, \tilde T]; H^{s}\right)$ and from Remark \ref{prodhs} and Remark \ref{div}  we obtain,
\begin{align*}
\| (\mathbf u \cdot \nabla) \mathbf \theta\|^2_{L^{\infty}\left([0, \tilde T]; H^{s-1}\right)} &= \left(\sup_{t \in [0, \tilde T]}\| (\mathbf u \cdot \nabla) \mathbf \theta\|_{H^{s-1}}\right)^2
\\ &\leq \left(\sup_{t \in [0, \tilde T]}\| \mathbf u \|_{\H^s} \| \mathbf \theta\|_{H^{s}}\right)^2
\\ &\leq \left(\sup_{t \in [0, \tilde T]}\| \mathbf u \|_{\H^s} \cdot \sup_{t \in [0, \tilde T]}\| \mathbf \theta\|_{H^{s}}\right)^2 < C < \infty.
\end{align*}
So from this result for $\left.\right.|t_2 - t_1| < \frac{\epsilon}{C {2^{2N+1}}}$,
\begin{align}\label{ile}
 I \leq C \sum_{j < N} 2^{2j} \left|t_2 - t_1\right| \leq C 2^{2N} \left|t_2 - t_1\right| < \frac{\epsilon}{2}. \left.\right.\left.\right. 
\end{align}
So by combining $\eqref{tle}$ and $\eqref{ile}$ we capture $\mathbf \theta \in C\left([0, \tilde T]; H^{s}(\mathbb R^n)\right)$. This completes the proof.
\end{proof}

\subsection{Blow-up Criterion}\label{BC}

In this subsection we will prove the Blowup criterion of the local-in-time strong solutions of  $B_{\nu, 0}$. Here we keep our attention to the three-dimensions, as global solvability in two-dimensions for the system $B_{\nu, 0}$ is known due to Chae \cite{Dc}. We show that the $BMO$ norms of the vorticity and gradient of temperature controls the breakdown of smooth solutions. Later we prove that the condition on the gradient of temperature can be relaxed under suitable assumption on the regularity of the initial temperature. We here assume $f = e_3,$ where $e_3$ denotes the $3^{\textrm{rd}}$ standard basis vector in $\mathbb{R}^3$, i.e., $e_3 = (0,0,1)$.

\begin{theorem}\label{buc}
Let $(\mathbf u_{0}, \mathbf \theta_{0}) \in \H^s(\mathbb R^n) \times H^{s}(\mathbb R^n)$, s $>$ $\frac{n}{2}$+1, n = 3. Let $(\mathbf u, \mathbf \theta) \in L^2\left([0, \tilde T]; \H^{s+1}(\mathbb R^n)\right) \cap C\left([0, \tilde T]; \H^{s}(\mathbb R^n)\right) \times C\left([0, \tilde T]; H^{s}(\mathbb R^n)\right)$ be a strong solution of Boussinesq equations. If $(\mathbf u, \mathbf \theta)$ satisfies the condition 
\begin{align}\label{buce}
\int_{0}^{\tilde T} \left( \|\nabla \times \mathbf u(\tau)\|_{BMO} + \|\nabla \mathbf \theta(\tau)\|_{BMO}\right) \, d\tau < \infty,
\end{align}
then $(\mathbf u, \mathbf \theta)$ can be continuously extended to $[0, T]$ for some $T > \tilde T.$ 
\end{theorem}
\begin{proof}
First apply $J^s$ to the system $B_{\nu, 0}$ to get
\begin{align}\label{jsu}
\frac{\partial (J^s \mathbf u)}{\partial t} + J^s \left[ (\mathbf u \cdot \nabla)\mathbf u \right] - \nu \Delta J^s \mathbf u + \nabla J^s p = J^s (\theta e_3),
\end{align}
\begin{align}\label{jst}
\frac{\partial (J^s \mathbf \theta)}{\partial t} + J^s \left[ (\mathbf u \cdot \nabla)\mathbf \theta \right] = 0
\end{align}
Taking the ${L^2}$-inner product of \eqref{jsu} with $J^s \mathbf u,$ and that of \eqref{jst} with $J^s \mathbf \theta,$ we have
\begin{align}\label{jip1}
\left( \frac{\partial (J^s \mathbf u)}{\partial t}, J^s \mathbf u \right)_{L^2} &+ \left( J^s \left[ (\mathbf u \cdot \nabla)\mathbf u \right], J^s \mathbf u \right)_{L^2} - \left( \nu \Delta J^s \mathbf u, J^s \mathbf u \right)_{L^2} \nonumber
\\ &+ \left( \nabla J^s p, J^s \mathbf u \right)_{L^2} = \left( J^s (\theta e_3), J^s \mathbf u \right)_{L^2} 
\end{align}
\begin{align}\label{jip2}
\left( \frac{\partial (J^s \mathbf \theta)}{\partial t}, J^s \mathbf \theta \right)_{L^2} + \left( J^s \left[ (\mathbf u \cdot \nabla)\mathbf \theta \right], J^s \mathbf \theta \right)_{L^2} = 0
\end{align}
We estimate each term separately. \\
\noindent 
First note,
\begin{align*}
\left( \frac{\partial (J^s \mathbf u)}{\partial t}, J^s \mathbf u \right)_{L^2} = \frac{1}{2} \frac{d}{dt} \|J^s \mathbf u\|^{2}_{\L^2} = \frac{1}{2} \frac{d}{dt} \|\mathbf u\|^{2}_{\H^s}.
\end{align*}
Using divergence free condition on $\mathbf u$ and the commutator estimate \eqref{ce} we obtain,
\begin{align*}
&\left| \left( J^s \left[ (\mathbf u \cdot \nabla)\mathbf u \right], J^s \mathbf u \right)_{L^2}\right| \\
&\ \ \ \ = \left| \left( [J^s, \mathbf u] \nabla \mathbf u, J^s \mathbf u \right)_{L^2} + \left( (\mathbf u \cdot \nabla)J^s \mathbf u, J^s \mathbf u \right)_{L^2}\right| = \left| \left( [J^s, \mathbf u] \nabla \mathbf u, J^s \mathbf u \right)_{L^2}\right|
\\ & \ \ \ \ \leq \| [J^s, \mathbf u] \nabla \mathbf u\|_{\L^2} \| J^s \mathbf u \|_{\L^2} \leq C\|\nabla \mathbf u\|_{L^{\infty}}\| \mathbf u\|_{\H^s} \| \mathbf u\|_{\H^s}
\\ & \ \ \ \ \leq C  \left(\| \mathbf u\|^{2}_{\H^s} + \| \mathbf \theta\|^{2}_{H^s}\right) \left(\|\nabla \mathbf u\|_{L^{\infty}} + \|\nabla \mathbf \theta\|_{L^{\infty}}\right)
\end{align*}
Integration by parts yields,
\begin{align*}
\left( -\nu \Delta J^s \mathbf u, J^s \mathbf u \right)_{L^2} = \nu \|J^s \nabla \mathbf u\|^{2}_{\L^2} = \nu \|\nabla \mathbf u\|^{2}_{\H^s}.
\end{align*}
Similarly,
\begin{align*}
\left( \nabla J^s p, J^s \mathbf u \right)_{L^2} = - \left(J^s p, J^s \nabla \cdot \mathbf u \right)_{L^2} = 0.
\end{align*}
It is straightforward to see that
\begin{align*}
\left|\left( J^s (\theta e_3), J^s \mathbf u \right)_{L^2}\right| &\leq \| J^s (\theta e_3)\|_{L^2} \| J^s \mathbf u \|_{\L^2}
\\ &\leq \| \mathbf \theta\|_{H^s}\|\mathbf u\|_{\H^s} \leq C\left(\| \mathbf \theta\|^2_{H^s} + \|\mathbf u\|^2_{\H^s}\right)
\end{align*}
Similarly
\begin{align*}
\left( \frac{\partial (J^s \mathbf \theta)}{\partial t}, J^s \mathbf \theta \right)_{L^2} = \frac{1}{2} \frac{d}{dt} \|\mathbf \theta\|^{2}_{H^s},
\end{align*}
and divergence free condition on $\mathbf u$ and commutator estimate \eqref{ce} yield,
\begin{align*}
&\left| \left( J^s \left[ (\mathbf u \cdot \nabla)\mathbf \theta \right], J^s \mathbf \theta \right)_{L^2}\right| \\
&\ \ \ \ = \left| \left( [J^s, \mathbf u] \nabla \mathbf \theta, J^s \mathbf \theta \right)_{L^2} + \left( (\mathbf u \cdot \nabla)J^s \mathbf \theta, J^s \mathbf \theta \right)_{L^2}\right| = \left| \left( [J^s, \mathbf u] \nabla \mathbf \theta, J^s \mathbf \theta \right)_{L^2}\right|
\\ &\ \ \ \ \leq \| [J^s, \mathbf u] \nabla \mathbf \theta\|_{L^2} \| J^s \mathbf \theta \|_{L^2} \leq C\left( \|\nabla \mathbf u\|_{L^{\infty}}\|\mathbf \theta\|_{H^{s}} + \|\mathbf u\|_{\H^s}\|\nabla \mathbf \theta\|_{L^{\infty}}\right) \|\mathbf \theta\|_{H^{s}}
\\ &\ \ \ \ \leq C\left( \|\nabla \mathbf u\|_{L^{\infty}}\|\mathbf \theta\|^{2}_{H^{s}} + \|\mathbf u\|_{\H^s}\|\nabla \mathbf \theta\|_{L^{\infty}} \|\mathbf \theta\|_{H^{s}}\right)
\\ &\ \ \ \ \leq C\left( \|\nabla \mathbf u\|_{L^{\infty}}\left(\|\mathbf u\|^{2}_{\H^{s}} + \|\mathbf \theta\|^{2}_{H^{s}}\right)+ \|\nabla \mathbf \theta\|_{L^{\infty}} \left( \|\mathbf u\|^{2}_{\H^s} + \|\mathbf \theta\|^{2}_{H^{s}}\right)\right)
\\ &\ \ \ \ \leq C \left(\|\mathbf u\|^{2}_{\H^{s}} + \|\mathbf \theta\|^{2}_{H^{s}}\right) \left(\|\nabla \mathbf u\|_{L^{\infty}}+ \|\nabla \mathbf \theta\|_{L^{\infty}}\right).
\end{align*}
Adding \eqref{jip1} and \eqref{jip2}, and using all the above estimates, we obtain,
\begin{align*}
&\frac{1}{2} \frac{d}{dt} \left(\|\mathbf u\|^{2}_{\H^s} + \|\mathbf \theta\|^{2}_{H^s}\right) + \nu \|\nabla \mathbf u\|^{2}_{\H^s} \\
&\ \ \ \ \leq C \left(\|\nabla \mathbf u\|_{L^{\infty}}+ \|\nabla \mathbf \theta\|_{L^{\infty}} + 1 \right) \left(\|\mathbf u\|^{2}_{\H^{s}} + \|\mathbf \theta\|^{2}_{H^{s}}\right).
\end{align*}
Ignoring the second term on the left hand side of the above estimate, we have after rearrangement,
\[ \frac{d}{dt} \left(\|\mathbf u\|^{2}_{\H^s} + \|\mathbf \theta\|^{2}_{H^s}\right) \leq C \left(\|\nabla \mathbf u\|_{L^{\infty}}+ \|\nabla \mathbf \theta\|_{L^{\infty}} + 1 \right) \left(\|\mathbf u\|^{2}_{\H^{s}} + \|\mathbf \theta\|^{2}_{H^{s}}\right).\]
Setting $Z(t) = \|\mathbf u(t)\|^{2}_{\H^s} + \|\mathbf \theta(t)\|^{2}_{H^s}$ for $t \in [0, \tilde T]$, we obtain,
\[\frac{d}{dt} Z(t) \leq C \left(\|\nabla \mathbf u(t)\|_{L^{\infty}}+ \|\nabla \mathbf \theta(t)\|_{L^{\infty}} + 1 \right) Z(t)\]
Standard Gronwall's inequality gives
\[ Z(t) \leq Z(0) \left.\right. exp \left( C \int_{0}^{t} (\|\nabla \mathbf u(\tau)\|_{L^{\infty}}+ \|\nabla \mathbf \theta(\tau)\|_{L^{\infty}} + 1) \, d\tau \right).\]
Hence
\begin{align}\label{ee3}
 &\|\mathbf u(t)\|^{2}_{\H^s} + \|\mathbf \theta(t)\|^{2}_{H^s} \nonumber
\\ &\ \ \ \ \leq \left(\|\mathbf u_0\|^{2}_{\H^s} + \|\mathbf \theta_0\|^{2}_{H^s}\right) \left.\right. exp \left( C \int_{0}^{t} (\|\nabla \mathbf u(\tau)\|_{L^{\infty}}+ \|\nabla \mathbf \theta(\tau)\|_{L^{\infty}} + 1) \, d\tau \right).
\end{align}
Due to the logarithmic Sobolev inequality given in Lemma \ref{kte}, and the fact that singular integral operators of Calderon-Zygmund type are bounded in $BMO$ (i.e. $\|\nabla \mathbf u\|_{BMO} \leq \|\nabla \times \mathbf u\|_{BMO}$), we have, for $s > \frac{n}{2} + 1$, 
\begin{align}\label{bmou}
\|\nabla \mathbf u\|_{L^{\infty}} &\leq C\left( 1 + \|\nabla \mathbf u\|_{BMO}\left(1 + log^{+}\|\nabla \mathbf u\|_{ \H^{s-1}}\right) \right) \nonumber
\\ &\leq C\left( 1 + \|\nabla \times \mathbf u\|_{BMO}\left(1 + log^{+}\|\mathbf u\|_{\H^{s}}\right) \right) \nonumber
\\ &\leq C\left( 1 + \|\nabla \times \mathbf u\|_{BMO}\left(1 + \frac{1}{2}log^{+}\|\mathbf u\|^{2}_{\H^{s}}\right) \right) \nonumber
\\ &\leq C\left( 1 + \|\nabla \times \mathbf u\|_{BMO}\left(1 + \frac{1}{2}log^{+} \left(\|\mathbf u\|^{2}_{\H^{s}} + \|\mathbf \theta\|^{2}_{H^{s}}\right) \right) \right) \nonumber
\\ &\leq C\left( 1 + \|\nabla \times \mathbf u\|_{BMO}\left(1 + log^{+} \left(\|\mathbf u\|^{2}_{\H^{s}} + \|\mathbf \theta\|^{2}_{H^{s}}\right) \right) \right).
\end{align}
Similarly we obtain for $\nabla \mathbf \theta$,
\begin{align}\label{bmot}
\|\nabla \mathbf \theta\|_{L^{\infty}} \leq C\left( 1 + \|\nabla \mathbf \theta\|_{BMO}\left(1 + log^{+} \left(\|\mathbf u\|^{2}_{\H^{s}} + \|\mathbf \theta\|^{2}_{H^{s}}\right) \right) \right).
\end{align}
Now using \eqref{bmou} and \eqref{bmot} in \eqref{ee3}, we obtain for all $t \in [0, \tilde T]$,
\begin{align*}
&\|\mathbf u(t)\|^{2}_{\H^s} + \|\mathbf \theta(t)\|^{2}_{H^s}\\
& \ \ \ \ \leq\left(\|\mathbf u_0\|^{2}_{\H^s} + \|\mathbf \theta_0\|^{2}_{H^s}\right) exp\Big[C\int_{0}^{t} \Big\{3 + \left(\|\nabla \times \mathbf u(\tau)\|_{BMO}+ \|\nabla \mathbf \theta(\tau)\|_{BMO}\right)
\\ & \ \ \ \ \quad \times\left(1 + log^{+} \left(\|\mathbf u(\tau)\|^{2}_{\H^{s}} + \|\mathbf \theta(\tau)\|^{2}_{H^{s}}\right) \right)\, \Big\}d\tau\Big].
\end{align*}
Let $X(t) = \|\mathbf u(t)\|^{2}_{\H^s} + \|\mathbf \theta(t)\|^{2}_{H^s}.$ Then by taking ``log" on both sides we get for all $t \in [0, \tilde T]$,
\begin{align*}
&\log X(t) \\
&\ \ \ \ \leq \log X(0) + C \int_{0}^{t} \Big\{3 + \left(\|\nabla \times \mathbf u(\tau)\|_{BMO}+ \|\nabla \mathbf \theta(\tau)\|_{BMO}\right) (1 + log^{+} X(\tau))\Big\} \, d\tau.
\end{align*}
Rearranging the terms we have
\begin{align*}
&\log (eX(t)) \\
&\ \ \ \ \leq \log (eX(0)) + C\tilde{T} + \int_{0}^{t} \Big\{\left(\|\nabla \times \mathbf u(\tau)\|_{BMO}+ \|\nabla \mathbf \theta(\tau)\|_{BMO}\right) (log (eX(\tau)))\Big\} \, d\tau.
\end{align*}
 Now Gronwall's inequality yields
\[log (eX(t)) \leq (log (eX(0)) + C\tilde T)\left. exp\right.\left( C \int_{0}^{t} (\|\nabla \times \mathbf u(\tau)\|_{BMO}+ \|\nabla \mathbf \theta(\tau)\|_{BMO}) \, d\tau \right).\]
Taking supremum over all $t \in [0, \tilde T]$ we obtain,
\begin{align*}
&\sup_{t \in [0, \tilde T]} log X(t) \leq \sup_{t \in [0, \tilde T]} log (eX(t))
\\ & \ \ \ \ \leq (log (eX(0)) + C\tilde T)\left. exp\right.\left( C \int_{0}^{\tilde T} (\|\nabla \times \mathbf u(\tau)\|_{BMO}+ \|\nabla \mathbf \theta(\tau)\|_{BMO}) \, d\tau \right).
\end{align*}
So finally we acquire,
\begin{align*}
&\sup_{t \in [0, \tilde T]} X(t) \\
& \ \ \ \ \leq e^{(1+C\tilde T)} X(0)exp\left\{ exp\left( C \int_{0}^{\tilde T} (\|\nabla \times \mathbf u(\tau)\|_{BMO}+ \|\nabla \mathbf \theta(\tau)\|_{BMO}) \, d\tau \right)\right\}.
\end{align*}
This concludes that if 
\[\int_{0}^{\tilde T} (\|\nabla \times \mathbf u(\tau)\|_{BMO}+ \|\nabla \mathbf \theta(\tau)\|_{BMO}) \, d\tau < \infty,\]
then by continuation of local solutions, we can extend the solution to $[0, T]$ for some $T > \tilde T$.
\end{proof}

\begin{remark}
Use of the logarithmic Sobolev inequality given in Lemma \ref{lsi} in the above proof will replace the condition \eqref{buce} by \[\int_{0}^{\tilde T} (\|\nabla \times \mathbf u(\tau)\|_{L^{\infty}}+ \|\nabla \mathbf \theta(\tau)\|_{L^{\infty}}) \, d\tau < \infty.\]
\end{remark}
We now show that the assumption on $\nabla\mathbf\theta$, as made in Theorem \ref{buc}, can be relaxed completely. In other words, the bound on curl of $\mathbf u$ is enough to extend the solution continuously to some time $T > \tilde T$, provided $ \mathbf \theta_0 \in H^s(\mathbb R^{n}) \cap W^{1, p}(\mathbb R^{n})$.\\
\noindent
Before proving the above result let us note the following vector identity.
\begin{remark}\label{vi}
\begin{align*}
\nabla (\mathbf u \cdot \nabla \mathbf \theta) &= (\mathbf u \cdot \nabla) \nabla \mathbf \theta + (\nabla \mathbf \theta \cdot  \nabla) \mathbf u + \mathbf u \times (\nabla \times \nabla \mathbf \theta) + \nabla \mathbf \theta \times (\nabla \times \mathbf u)
\\ &= (\mathbf u \cdot \nabla) \nabla \mathbf \theta + (\nabla \mathbf \theta \cdot  \nabla) \mathbf u + \nabla \mathbf \theta \times (\nabla \times \mathbf u)
\\ &=  (\mathbf u \cdot \nabla) \nabla \mathbf \theta + (\nabla \mathbf u)^{t} \cdot \nabla \mathbf \theta 
\end{align*}
where we have used the facts that curl of the gradient of a scalar function is zero (i.e., $\mathbf u \times (\nabla \times \nabla \mathbf \theta) = 0$)  and $(\nabla \mathbf u)^{t} \cdot \nabla \mathbf \theta = (\nabla \mathbf \theta \cdot  \nabla) \mathbf u + \nabla \mathbf \theta \times (\nabla \times \mathbf u).$.
\end{remark}

\begin{theorem}\label{buc1}
Let $s > \frac{n}{2} + 1$, $\mathbf u_{0} \in \H^{s}(\mathbb R^{n}),$ and $ \mathbf \theta_0 \in H^s(\mathbb R^{n}) \cap W^{1, p}(\mathbb R^{n}),$ for $ 2 \leq p \leq \infty,$ n=3. Let $(\mathbf u, \mathbf \theta) \in L^2\left([0, \tilde T]; \H^{s+1}(\mathbb R^n)\right) \cap C\left([0, \tilde T]; \H^{s}(\mathbb R^n)\right) \times C\left([0, \tilde T]; H^{s}(\mathbb R^n)\right)$ as before. Then 
\[\int_{0}^{\tilde T} \|\nabla \times \mathbf u(\tau)\|_{BMO} \,d\tau < \infty\]
guarantees that the solution can be extended continuously to $[0, T]$ for some $T > \tilde T.$
\end{theorem}

\begin{proof}
We consider the equation \eqref{B2} i.e.,
\[\frac{\partial \mathbf \theta}{\partial t} + (\mathbf u \cdot \nabla)\mathbf \theta = 0,\]
and apply the gradient operator $\nabla = (\partial_{x_1}, \dots, \partial_{x_n})$ on both sides and take $L^2$-inner product with $\nabla \mathbf \theta |\nabla \mathbf \theta|^{p-2}$ to obtain,
\[\left(\frac{\partial}{\partial t} (\nabla \mathbf \theta), \nabla \mathbf \theta |\nabla \mathbf \theta|^{p-2}\right) + \left(\nabla (\mathbf u \cdot \nabla \mathbf \theta), \nabla \mathbf \theta |\nabla \mathbf \theta|^{p-2}\right) = 0.\]
Using the vector identity in Remark \ref{vi} we obtain,
\begin{align}\label{tttt}
\left(\frac{\partial}{\partial t} (\nabla \mathbf \theta), \nabla \mathbf \theta |\nabla \mathbf \theta|^{p-2}\right) + \left( (\nabla \mathbf u)^{t} \cdot \nabla \mathbf \theta, \nabla \mathbf \theta |\nabla \mathbf \theta|^{p-2}\right) + \left( (\mathbf u \cdot \nabla) \nabla \mathbf \theta, \nabla \mathbf \theta |\nabla \mathbf \theta|^{p-2}\right) = 0.
\end{align}
The first term of \eqref{tttt} gives,
\[\left(\frac{\partial}{\partial t} (\nabla \mathbf \theta), \nabla \mathbf \theta |\nabla \mathbf \theta|^{p-2}\right) = \frac{1}{p} \int_{\mathbb R^{n}} \frac{\partial}{\partial t}  |\nabla \mathbf \theta|^{p} \,dx = \frac{1}{p} \frac{d}{dt} \| \nabla \mathbf \theta\|_{L^p}^{p}.\]
The second term of \eqref{tttt} yields,
\begin{align*}
\left( (\nabla \mathbf u)^{t} \cdot \nabla \mathbf \theta, \nabla \mathbf \theta |\nabla \mathbf \theta|^{p-2}\right) &= \int_{\mathbb R^{n}} (\nabla \mathbf u)^{t} \cdot \nabla \mathbf \theta \cdot \nabla \mathbf \theta |\nabla \mathbf \theta|^{p-2} \,dx 
\\ &\leq \int_{\mathbb R^{n}} (\nabla \mathbf u)^{t} \cdot |\nabla \mathbf \theta|^{p}
\leq \|\nabla \mathbf u\|_{L^{\infty}} \|\nabla \mathbf \theta\|_{L^p}^{p}.
\end{align*}
By applying integration by parts and the divergence free condition of $\mathbf u$, we have from the third term of \eqref{tttt},
\begin{align*}
\left( (\mathbf u \cdot \nabla) \nabla \mathbf \theta, \nabla \mathbf \theta |\nabla \mathbf \theta|^{p-2}\right) &= \int_{\mathbb R^{n}} (\mathbf u \cdot \nabla) \nabla \mathbf \theta \cdot \nabla \mathbf \theta |\nabla \mathbf \theta|^{p-2} \,dx
\\ &= \frac{1}{p}  \int_{\mathbb R^{n}} \mathbf u \cdot \nabla |\nabla \mathbf \theta|^{p} \,dx
= -\frac{1}{p} \int_{\mathbb R^{n}} (\nabla \cdot \mathbf u) \cdot |\nabla \mathbf \theta|^{p} \,dx = 0.
\end{align*}
So from the term-wise estimates of \eqref{tttt}, we obtain,
\[\frac{d}{dt} \| \nabla \mathbf \theta\|_{L^p}^{p} \leq  p \|\nabla \mathbf u\|_{L^{\infty}} \|\nabla \mathbf \theta\|_{L^p}^{p},\]
which further gives due to Gronwall's inequality,
\[ \| \nabla \mathbf \theta\|_{L^p}^{p} \leq  \| \nabla \mathbf \theta_{0}\|_{L^p}^{p} \left.exp\right.\left( p \int_{0}^{t} \|\nabla \mathbf u(\tau)\|_{L^{\infty}} \,d\tau \right). \]
So we have,
\[ \| \nabla \mathbf \theta\|_{L^p} \leq  \| \nabla \mathbf \theta_{0}\|_{L^p} \left.exp\right.\left( \int_{0}^{t} \|\nabla \mathbf u(\tau)\|_{L^{\infty}} \,d\tau \right). \]
Letting $p \to \infty,$
\begin{align*}
\| \nabla \mathbf \theta\|_{L^\infty} \leq  \| \nabla \mathbf \theta_{0}\|_{L^\infty} \left.exp\right.\left( \int_{0}^{t} \|\nabla \mathbf u(\tau)\|_{L^{\infty}} \,d\tau \right).
\end{align*}
Due to Lemma \ref{kte}, and properties of $BMO$ spaces, we further have,
\begin{align*}
\| \nabla \mathbf \theta\|_{L^\infty} \leq  \| \nabla \mathbf \theta_{0}\|_{L^\infty} \left.exp\right.\left( C \int_{0}^{t} \left( 1 + \|\nabla \times \mathbf{u}(\tau)\|_{BMO}\left(1 + log^{+}\|\mathbf{u}(\tau)\|_{\H^{s}}\right) \right) \,d\tau \right).
\end{align*}
 As  $ \mathbf \theta_{0} \in H^s(\mathbb R^{n}) \cap W^{1, p}(\mathbb R^{n}), 2\leq p \leq\infty$ and $\sup_{t \in [0, \tilde T]} \|\mathbf u\|_{\H^{s}}$ is bounded for $s>n/2+1$, we have,
\begin{align}\label{tw}
\| \nabla \mathbf \theta\|_{L^\infty} \leq C \left.exp\right. \left(\int_{0}^{\tilde T}  \|\nabla \times \mathbf{u}(\tau)\|_{BMO} \, d\tau\right) 
\end{align}
where $C = C(\|\nabla \mathbf \theta_{0}\|_{L^\infty}, \|\mathbf u\|_{\H^{s}}, \tilde T)$. \\
\noindent
Due to the assumption $\int_{0}^{\tilde T} \|\nabla \times \mathbf u(\tau)\|_{BMO} \,d\tau < \infty$, the estimate in \eqref{tw} is bounded. Hence, $\| \nabla \mathbf \theta\|_{BMO} \leq 2 \| \nabla \mathbf \theta\|_{L^\infty} \leq C < \infty.$
So the bound on BMO norm of curl of $ \mathbf u$ is enough to guarantee that the solution can be extended to $[0, T]$ for some $T > \tilde T$ provided $ \mathbf \theta_{0} \in H^s(\mathbb R^{n}) \cap W^{1, p}(\mathbb R^{n}).$ 
\end{proof}

\section{Inviscid ($B_{0, \kappa}$) and Ideal  ($B_{0, 0}$) Boussinesq Systems}
In this section we will focus on the Boussinesq systems $B_{0, \kappa}$ and $B_{0, 0}$, and build similar results as for $B_{\nu, 0}$ in the previous section. 

The inviscid Boussinesq system ($B_{0, \kappa}$) for the incompressible fluid flows interacting with an active scalar is given by:
\begin{align}\label{Bi}
 \frac{\partial \mathbf u}{\partial t} + (\mathbf u \cdot \nabla)\mathbf u + \nabla p & = \mathbf \theta f,  \ \ \textrm{in}\ \mathbb{R}^n\times (0, \infty)\\
\frac{\partial \mathbf \theta}{\partial t} + (\mathbf u \cdot \nabla)\mathbf \theta & = \kappa \Delta \mathbf \theta,  \ \ \textrm{in}\ \mathbb{R}^n\times (0, \infty)\\
\nabla \cdot \mathbf u & = 0,  \ \ \textrm{in}\ \mathbb{R}^n\times (0, \infty)\\
\mathbf u(x,0) = \mathbf u_{0}(x), & \ \mathbf \theta(x,0) = \mathbf \theta_{0}(x ),  \ \ \textrm{in}\ \mathbb{R}^n, \label{Bii}
\end{align}
where $n$ = 2, 3.

We have the following existence and uniqueness result.
\begin{theorem}\label{NVBE1}
Let $s > \frac{n}{2}+1,$ $\mathbf u_{0} \in \H^s(\mathbb R^{n})$, $\mathbf \theta_{0} \in  H^s(\mathbb R^{n})$, and $f\in L^{\infty} ([0, T];$ $H^s(\mathbb R^n)),$ $n = 2, 3.$ Then there exists a unique solution $(\mathbf u, \mathbf \theta) \in C([0, \tilde{T}]; \H^{s}(\mathbb R^{n}))$ $\times C([0, \tilde{T}]; H^{s}(\mathbb R^{n})) \cap L^{2}([0, \tilde{T}];  H^{s+1}(\mathbb R^{n}))$ to the equations \eqref{Bi}-\eqref{Bii} for some finite time $\tilde{T} = \tilde{T}(s, \kappa, \|u_0\|_{\H^{s}}, \|\theta_0\|_{H^{s}}) > 0$, satisfying the energy estimates
\[ \sup_{t \in [0, \tilde T]} \| \mathbf u(t)\|_{\H^s}<\infty, \left. \left. \left. \right. \right. \right. \sup_{t \in [0, \tilde T]}\| \mathbf \theta(t)\|_{H^s}<\infty, \left. \left. \left. \right. \right. \right. \int_{0}^{\tilde T} \| \nabla \mathbf \theta(t)\|^2_{H^s} \,dt<\infty\]

\end{theorem}

\begin{proof}
The idea of the proof is similar to Theorem\ref{mt1} for the case $(B_{\nu, 0}).$ However, one needs to carefully handle the estimates concerning the non-linear terms due to the lack of bounds for $\|\nabla\mathbf u\|_{\H^s}$. Needless to say, in Proposition \ref{fin}, calculations concerning the linear terms remain same. Note that, as $s>n/2+1$,
\begin{align*}
 \left( \Lambda^s \left[ (\mathbf u \cdot \nabla)\mathbf \theta \right], \Lambda^s \mathbf \theta \right)_{L^2} &\leq \| \Lambda^s (\mathbf u \cdot \nabla \mathbf \theta) \|_{L^2} \| \Lambda^s \mathbf \theta \|_{L^2}
\leq \| \mathbf u \cdot \nabla \mathbf \theta \|_{H^s} \|\mathbf \theta \|_{H^s}
\\ &\leq  \| \mathbf u \|_{\H^s}  \| \nabla \mathbf \theta \|_{ H^s} \|\mathbf \theta \|_{H^s}
\leq \frac{1}{2} \| \nabla \mathbf \theta \|_{ H^s} \left( \| \mathbf u \|^2_{\H^s} + \|\mathbf \theta \|^2_{H^s}\right)
\\ &\leq \frac{1}{2} \left(\frac{1}{4 \kappa} \left( \| \mathbf u \|^2_{\H^s} + \|\mathbf \theta \|^2_{H^s}\right)^{2} + \kappa  \| \nabla \mathbf \theta \|^2_{ H^s}\right),
\end{align*}
and
\begin{align*}
 \left( \Lambda^s \left[ (\mathbf u \cdot \nabla)\mathbf u \right], \Lambda^s \mathbf u \right)_{L^2} &=  \left( \left[ \Lambda^s, \mathbf u \right]\nabla \mathbf u , \Lambda^s \mathbf u \right)_{L^2} + \left( (\mathbf u \cdot \nabla) \Lambda^s \mathbf u, \Lambda^s \mathbf u \right)_{L^2} \nonumber
\\ &=  \left( \left[ \Lambda^s, \mathbf u \right]\nabla \mathbf u , \Lambda^s \mathbf u \right)_{L^2} 
\leq \| \left[ \Lambda^s, \mathbf u \right]\nabla \mathbf u \|_{\L^2} \|\Lambda^s \mathbf u \|_{\L^2} \nonumber\\
&\leq C \|\nabla \mathbf u\|_{L^\infty} \| \mathbf u \|_{\H^s} \| \mathbf u \|_{\H^s} \nonumber
 \\&\leq C  \|\nabla \mathbf u\|_{\H^{s-1}} \| \mathbf u \|^{2}_{\H^s} = C  \| \mathbf u \|^{3}_{\H^s}.
\end{align*}
Thus the $H^s$-energy estimate will take the form
\begin{align}\label{eg1}
\frac{d}{dt} &\left( \| \mathbf u \|^{2}_{\H^s} + \| \mathbf \theta \|^{2}_{ H^s}\right) + \kappa \| \nabla \mathbf \theta \|^{2}_{ H^s}\nonumber
\\ &\leq C \| \mathbf u \|^{3}_{\H^s} + \frac{C}{4 \kappa} \left( \| \mathbf u \|^{2}_{\H^s} + \| \mathbf \theta \|^{2}_{ H^s}\right)^{2} + C \| f \|_{ H^s} \left( \| \mathbf u \|^{2}_{\H^s} + \| \mathbf \theta \|^{2}_{ H^s}\right).
\end{align}
The cubic term can be rewritten, using Young's inequality, as
\begin{align*}
 \| \mathbf u \|^{3}_{\H^s} =  \| \mathbf u \|_{\H^s} \| \mathbf u \|^{2}_{\H^s}  &\leq  \| \mathbf u \|_{\H^s} \left(\| \mathbf u \|^{2}_{\H^s} + \| \mathbf \theta \|^{2}_{ H^s}\right)
\\ &\leq C\left[  \left(\| \mathbf u \|^{2}_{\H^s} + \| \mathbf \theta \|^{2}_{ H^s} \right) +  \left( \| \mathbf u \|^{2}_{\H^s} + \| \mathbf \theta \|^{2}_{ H^s}\right)^{2}\right] .
\end{align*}
Letting $X = \| \mathbf u \|^{2}_{\H^s} + \| \mathbf \theta \|^{2}_{ H^s},$ finally we have from \eqref{eg1}, after rearrangement of terms,
\begin{align*}
\frac{dX}{dt} + \kappa \| \nabla \mathbf \theta \|^{2}_{ H^s} \leq C+C X +C(\frac{1}{4 \kappa}+1) X^2.
\end{align*}
The rest of the proof for energy estimate follows from Bihari's inequality and is similar Proposition \ref{fin}, and the blowup time depends on $\kappa$. The proofs of the results similar to Propositions \ref{Ca} - \ref{p8} can be done with certain modifications from place to place.
\end{proof}
\noindent
The blow up criteria for the system $B_{0, \kappa}$ in three-dimensions can be stated as follows, and the line of proof is exactly similar to Theorem \ref{buc} and Theorem \ref{buc1}.
\begin{theorem}\label{NVBE2}
Let $s > \frac{n}{2}+1, f = e_n, \mathbf u_{0} \in \H^s(\mathbb R^{n})$, $\mathbf \theta_{0} \in  H^s(\mathbb R^{n}) \cap W^{1, p}(\mathbb R^{n})$ and $ 2 \leq p \leq \infty,$  n=3. Let  $(\mathbf u, \mathbf \theta) \in C([0, \tilde{T}]; \H^{s}(\mathbb R^{n})) \times C([0, \tilde{T}]; H^{s}(\mathbb R^{n})) \cap L^{2}([0, \tilde{T}];  H^{s+1}(\mathbb R^{n}))$ be a solution of the equations \eqref{Bi}-\eqref{Bii}. If $\mathbf u$ satisfies the condition
\[\int_{0}^{\tilde T} \|\nabla \times \mathbf{u}(\tau)\|_{BMO} \,d\tau < \infty,\]
then $(\mathbf u, \mathbf \theta)$ can be continuously extended to $[0, T]$ for some $T > \tilde T.$
\end{theorem}

Now we consider the ideal Boussinesq system $B_{0, 0}$, i.e., when both kinematic viscosity and thermal diffusivity are zero,
\begin{align}\label{Bid}
 \frac{\partial \mathbf u}{\partial t} + (\mathbf u \cdot \nabla)\mathbf u + \nabla p &= \mathbf \theta f,\\
\frac{\partial \mathbf \theta}{\partial t} + (\mathbf u \cdot \nabla)\mathbf \theta &= 0,\\
\nabla \cdot \mathbf u &= 0,\\
\mathbf u(x,0) = \mathbf u_{0}(x), &\ \mathbf \theta(x,0) = \mathbf \theta_{0}(x ),\label{Bidi}
\end{align}
in the whole of $\mathbb R^{n}$ for $n$ = 2, 3.

The main result on the local-in-time existence and uniqueness is as follows:
\begin{theorem}\label{IBE1}
Let $s > \frac{n}{2}+1,$ $\mathbf u_{0} \in \H^s(\mathbb R^{n})$, $\mathbf \theta_{0} \in  H^s(\mathbb R^{n})$, and $f\in L^{\infty} ([0, T];$ $H^s(\mathbb R^n)),$ $n = 2, 3.$ Then there exists a unique solution $(\mathbf u, \mathbf \theta) \in C([0, \tilde{T}]; \H^{s}(\mathbb R^{n}))$ $\times C([0, \tilde{T}]; H^{s}(\mathbb R^{n}))$ to the equations \eqref{Bid}-\eqref{Bidi} for some finite time $\tilde{T} = \tilde{T}(s, \|u_0\|_{\H^{s}}, \|\theta_0\|_{H^{s}}) > 0$, satisfying the energy estimates
\[ \sup_{t \in [0, \tilde T]} \| \mathbf u(t)\|_{\H^s}<\infty, \left. \left. \left. \right. \right. \right. \sup_{t \in [0, \tilde T]}\| \mathbf \theta(t)\|_{H^s}<\infty.\]
\end{theorem}

\begin{proof}
Note that, in this case, there are no bounds on the $H^s$-norms of both $\nabla \mathbf u$ and $\nabla \mathbf \theta$. The estimate for the non-linear term $\left( \Lambda^s \left[ (\mathbf u \cdot \nabla)\mathbf u \right], \Lambda^s \mathbf u \right)_{L^2}$ remain same as in Theorem \ref{NVBE1}. However, the other non-linear term needs to be taken care separately. As $s>n/2+1$,
\begin{align*}
 \left( \Lambda^s \left[ (\mathbf u \cdot \nabla)\mathbf \theta \right], \Lambda^s \mathbf \theta \right)_{L^2} &\leq C\left( \| \nabla\mathbf u\|_{L^{\infty}}  \|\mathbf \theta\|_{H^s} + \|\mathbf u\|_{\H^s}  \|\nabla \mathbf \theta\|_{L^{\infty}} \right) \|\mathbf \theta\|_{H^s}
\\ &\leq C\left( \|\mathbf u\|_{\H^s}  \|\mathbf \theta\|_{H^s} + \|\mathbf u\|_{\H^s}  \|\mathbf \theta\|_{H^s} \right) \|\mathbf \theta\|_{H^s}
\\ &\leq 2 \|\mathbf u\|_{\H^s} \|\mathbf \theta\|^2_{H^s}.
\end{align*}
Therefore the energy estimate will take the form, after rearrangements,
\begin{align*}
\frac{d}{dt} &\left( \| \mathbf u \|^{2}_{\H^s} + \| \mathbf \theta \|^{2}_{ H^s}\right)
\\ &\leq C \| \mathbf u \|^{3}_{\H^s} + C\| \mathbf u \|_{\H^s}  \| \mathbf \theta \|^{2}_{ H^s} +  C\| f \|_{H^s} \left( \| \mathbf u \|^{2}_{\H^s} + \| \mathbf \theta \|^{2}_{ H^s}\right).
\end{align*}
The cubic term can be rewritten as in Theorem \ref{NVBE1}, and finally letting $X(t) = \| \mathbf{u}(t) \|^{2}_{\H^s} + \| \mathbf \theta(t) \|^{2}_{ H^s}$ and applying Young's inequality, we obtain
\begin{align*}
\frac{dX(t)}{dt} \leq CX(t) + CX(t)^{2}.
\end{align*}
Therefore application of Bihari's inequality as in Proposition \ref{fin} will provide the energy estimates and the blowup time. Proofs of the results similar to Propositions \ref{Ca} - \ref{p8} can be imitated for this case with certain modifications from place to place.
\end{proof}
There is no global well-posedness result available for the system $B_{0, 0}$ even in two-dimensions \cite{CN}. So the below Theorem is valid for both two and three dimensions and the line of proof is exactly similar to Theorem \ref{buc} and Theorem \ref{buc1}.
\begin{theorem}\label{IBE2}
Let $s > \frac{n}{2}+1, f = e_n,\mathbf u_{0} \in \H^s(\mathbb R^{n})$, $\mathbf \theta_{0} \in  H^s(\mathbb R^{n}) \cap W^{1, p}(\mathbb R^{n})$ and $ 2 \leq p \leq \infty,$  n=2, 3. Let  $(\mathbf u, \mathbf \theta) \in C([0, \tilde{T}]; \H^{s}(\mathbb R^{n})) \times C([0, \tilde{T}]; H^{s}(\mathbb R^{n}))$ be a solution of the equations \eqref{Bid}-\eqref{Bidi}. If $\mathbf u$ satisfies the condition
\[\int_{0}^{\tilde T} \|\nabla \times \mathbf{u} (\tau)\|_{BMO} \,d\tau < \infty,\]
then $(\mathbf u, \mathbf \theta)$ can be continuously extended to $[0, T]$ for some $T > \tilde T.$
\end{theorem}


\begin{thebibliography}{9}
\bibitem{AH} {\sc Abidi, H., Hmidi, T.} 
On the global well-posedness for Boussinesq system, 
{\em J. Differential Equations}, \textbf{233} (1), 199-220, 2007.

\bibitem{AHK} {\sc Abidi, H., Hmidi, T., Keraani, S.} 
On the global regularity of axisymmetric Navier-Stokes-Boussinesq system, 
{\em Discrete Contin. Dyn. Syst.}, \textbf{29} (3), 737-756, 2011.

\bibitem{AF} {\sc Adams, R. A., Fournier, J. J. F.} 
{\em Sobolev Spaces,}
Pure and Applied Mathematics (Amsterdam) Vol.\textbf{140}, 
Academic press, 1975.

\bibitem{BKM} {\sc  Beale, J.T., Kato, T., Majda,  A.}
 Remarks on the breakdown of smooth solutions for the 3-D Euler equations,
{\em Comm. Math. Phys.},
\textbf{94}, 61-66, 1984.

\bibitem{BF} {\sc Bessaih, H., Ferrario, B.}
The regularized 3D Boussinesq equations with fractional Laplacian and no diffusion,
arXiv:1504.05067v1 math.AP.

\bibitem{BG} {\sc Brezis, H., Gallouet, T.} 
Nonlinear Schr\"{o}dinger evolution equations, 
{\em Nonlinear Anal. TMA} 
\textbf{4}, 677-681, 1980.

\bibitem{BW} {\sc Brezis, H., Wainger, S.} 
A note on limiting cases of Sobolev embeddings and convolution inequalities, 
{\em Comm. Partial Differential Equations}
\textbf{5}, 773-789, 1980.

\bibitem{CD} {\sc Cannon, J.R., DiBenedetto, E.} The initial problem for the Boussinesq equations with data in $L^p$, {\em Lecture Note in Mathematics}, \textbf{771}, Springer, Berlin, 129-144, 1980.

\bibitem{Dc} {\sc  Chae, D.}
 Global regularity for the 2D Boussinesq equations with partial viscosity terms, 
{\em Advances in Mathematics},
\textbf{203}, 497-513, 2006.

\bibitem{CKN} {\sc  Chae, D., Kim, S.-K., Nam H.-S.}
 Local existence and blow-up criterion of H\"older continuous solutions of the Boussinesq equations, 
{\em Nagoya Math. J.},
\textbf{155}, 55-88, 1999. 

\bibitem{CN} {\sc  Chae, D., Nam H.-S.}
 Local existence and blow-up criterion for the Boussinesq equations, 
{\em Proc. of  Roy. Soc. Edinburgh, Sect. A},
\textbf{127} (5), 935-946, 1997.

\bibitem{DP} {\sc Danchin, R., Paicu, M.} 
Existence and uniqueness results for the Boussinesq system with data in Lorentz spaces, 
{\em Phys. D},
\textbf{237} (10-12), 1444-1460, 2008. 

\bibitem{DP2} {\sc Danchin, R., Paicu, M.} 
Les th\'{e}or\`{e}mes de Leray et de Fujita-Kato pour le syst\`{e}me de Boussinesq partiellement visqueux, 
{\em Bull. Soc. Math. France}
\textbf{136} (2), 261-309, 2008.

\bibitem{DP1} {\sc Danchin, R., Paicu, M.}  
Global well-posedness issues for the inviscid Boussinesq system with Yudovich's type data, 
{\em Comm. Math. Phys.}
\textbf{290} (1), 1-14, 2009.

\bibitem{DD} {\sc Dhongade, U. D., Deo, S. G.}  
A Nonlinear Generalization of Bihari's Inequality, 
{\em Proceedings of the American Mathematical Society}
\textbf{54} (1), 211-216, 1976.

\bibitem{Ev} {\sc Evans, L.C.}
{\em Partial Differential Equations}, Second Ed., Grad. Stud. Math., vol. 19, American Mathematical Society, Providence,
RI, 2010.

\bibitem{Fn} {\sc Fefferman, C. L.}
Characterizations of Bounded Mean Oscillation,
{\em Bulletin of the American Mathematical Society},
\textbf{77} (4), 587-588, 1971.

\bibitem{Fe} {\sc Fefferman, C. L., McCormick, D. S., Robinson, J. C. and Rodrigo, J. L.}
 Higher Order Commutator Estimates and Local Existence for the Non-resistive MHD Equations and Related Models,
{\em Journal of Functional Analysis},
\textbf{267}, 1035-1056, 2014.

\bibitem{GF} {\sc Geng, J.,  Fan, J.}
A note on regularity criterion for the 3D Boussinesq system with zero thermal conductivity,
{\em Appl. Math. Lett.}
\textbf{25} (1), 63-66, 2012.

\bibitem{HK} {\sc Hmidi,T., Keraani, S.} 
On the global well-posedness of the two-dimensional Boussinesq system with a zero diffusivity, 
{\em Adv. Differential Equations}
\textbf{12} (4), 461-480, 2007.

\bibitem{HL} {\sc Hou, T. Y., Li, C.}
Global Well-Posedness of the Viscous Boussinesq Equations,
{\em Discrete and Continuous  Dynamical Systems},
\textbf{12} (1), 1-12, 2005.

\bibitem{IM} {\sc Ishimura, N., Morimoto, H.,}
Remarks on the blow-up criterion for the 3-D Boussinesq equations,
{\em Mathematical Models and Methods in Applied Sciences}
\textbf{9} (9), 1323-1332, 1999.

\bibitem{KP} {\sc Kato, T., Ponce, G.}
Commutator estimates and the Euler and Navier-Stokes Equations,
{\em Comm. Pure Appl. Math.},
\textbf{41}, 891-907, 1988.

\bibitem{Ks} {\sc Kesavan, S.,}
{\em Topics in Functional Analysis and Applications}, Second Ed., 2015.

\bibitem{KT} {\sc Kozono, H., Taniuchi, Y.}
Limiting Case of the Sobolev Inequality in BMO, with Application to the Euler Equations,
{\em Comm. Math. Phys.},
\textbf{214}, 191-200, 2000.

\bibitem{LR} {\sc Lemari\'e-Rieusset P. G.}
{\em Recent developments in the Navier-Stokes problem}, Chapman \& Hall/CRC research in mathematics series,
431, 2002.

\bibitem{LM} {\sc Lions, J. L., Magenes, E.} 
{\em Non-homogeneous boundary Value problems and Applications,}
Vol.\textbf{1}, Springer-Verlag, 
New York, 1972.

\bibitem{Ro} {\sc Robinson, J. C.} 
{\em Infinite Dimensional Dynamical Systems, An Introduction to Dissipative Parabolic PDEs and the Theory of Global Attractors,}
Cambridge University Press, 
UK, 2001.

\bibitem{Te} {\sc Temam, R.} 
{\em Navier-Stokes equations. Theory and numerical analysis.} 
Studies in Mathematics and its Applications 2, North-Holland Publishing Co., 
Amsterdam-New York, 1979.

\bibitem{QDY} {\sc Qiu, H., Du, Y., Yao, Z.} 
A blow-up criterion for 3D Boussinesq equations in Besov spaces, 
{\em Nonlinear Anal.},
\textbf{73} (3), 806-815, 2010.

\bibitem{YX} {\sc Ye, Z., Xu, X.}
Global well-posedness of the 2D Boussinesq equations with fractional Laplacian dissipation,
arXiv:1506.00470v1 math.AP.

\bibitem{Ye} {\sc Ye, Z.}
Regularity criteria for 3D Boussinesq equations with zero thermal diffusion,
{\em Electronic Journal of Differential Equations}, 
\textbf{2015} (97), 1-7, 2015.

\bibitem{Yo} {\sc Yosida, K.} 
{\em Functional Analysis,} Sixth Ed.
Springer-Verlag, Berlin Heidelberg 
New York, 1980.



 
\end{thebibliography}
\end{document}